\documentclass[oneside,a4paper,11pt,reqno]{amsart}

\usepackage[T1]{fontenc}
\usepackage[utf8]{inputenc}
\usepackage{graphicx}

\usepackage{amsmath,hyperref}
\usepackage{amsfonts}
\usepackage{amssymb}
\usepackage{amsthm}
\usepackage{array}
\usepackage{color}
\usepackage{enumerate}
\usepackage{tikz}
\usepackage{tikz-cd}

\usetikzlibrary{arrows, automata}

\newtheorem{theorem}{Theorem}[section]

\newtheorem{corollary}[theorem]{Corollary}

\newtheorem{hypothesis}[theorem]{Hypothesis}
\newtheorem{lemma}[theorem]{Lemma}
\newtheorem{proposition}[theorem]{Proposition}

\newtheorem{remark}[theorem]{Remark}
\newtheorem{properties}[theorem]{Properties}

\numberwithin{equation}{section}

\newcommand{\N}{\mathbb{N}}
\newcommand{\Z}{\mathbb{Z}}

\newcommand{\R}{\mathbb{R}}
\newcommand{\C}{\mathbb{C}}

\newcommand{\Ebb}{\mathbb{E}}

\newcommand{\Lbb}{\mathbb{L}}

\newcommand{\Sbb}{\mathbb{S}}
\newcommand{\Tbb}{\mathbb{T}}

\newcommand{\Bcal}{\mathcal{B}}
\newcommand{\Ccal}{\mathcal{C}}

\newcommand{\Ecal}{\mathcal{E}}

\newcommand{\Gcal}{\mathcal{G}}

\newcommand{\Jcal}{\mathcal{J}}

\newcommand{\Lcal}{\mathcal{L}}
\newcommand{\Mcal}{\mathcal{M}}
\newcommand{\Ncal}{\mathcal{N}}

\newcommand{\Pcal}{\mathcal{P}}

\newcommand{\Zcal}{\mathcal{Z}}

\newcommand{\afrak}{\mathfrak{a}}

\newcommand{\norm}[2]{\left\| #1 \right\|_{#2}}
\newcommand{\dd}{\;{\rm d}}
\newcommand{\de}{{\rm d}}

\DeclareMathOperator{\Leb}{Leb}

\DeclareMathOperator{\Lip}{Lip}

\DeclareMathOperator{\Card}{Card}

\usepackage[top=3cm, bottom=2cm, left=1.5cm, right=1.5cm]{geometry}

\title{Central limit theorems for the $\Z^2$-periodic Lorentz gas}
\author{Fran\c{c}oise P\`ene}
\address{Univ Brest, Universit\'e de Brest, IUF, Institut Universitaire de France, LMBA,
Laboratoire de Math\'ematiques de Bretagne Atlantique, UMR CNRS 6205, 6 avenue Le Gorgeu, 29238 Brest Cedex, France}
\email{francoise.pene@univ-brest.fr}
\author{Damien Thomine}
\address{D\'epartement de Math\'ematiques d'Orsay, Universit\'e Paris-Sud, 
UMR CNRS 8628, F-91405 Orsay Cedex, France}
\email{damien.thomine@math.u-psud.fr}

\date\today

\begin{document}

\begin{abstract}
This paper is devoted to the study of the stochastic properties
of dynamical systems preserving an infinite measure. More precisely
we prove central limit theorems for Birkhoff sums of observables of
$\Z^2$-extensions of dynamical systems (satisfying some nice spectral properties).
The motivation of our paper is the $\Z^2$-periodic Lorentz process
for which we establish a functional central limit theorem for H\"older continuous observables
(in discrete time as well as in continuous time).
\end{abstract}

\maketitle

\section{Introduction}
A measure preserving dynamical system is given by a transformation $T$ or a flow $(Y_t)_{t\ge 0}$ preserving a measure.
When the measure is a probability, the study of the stochastic properties of such a dynamical system consists in studying the probabilistic properties of families of 
stationary
random variables of the form $(\phi\circ T^k)_{k\ge 0}$ or $(\phi\circ Y_t)_{t\ge 0}$
for reasonnable observables, with a particular interest in the study of the Birkhoff sums, which are given by
\begin{equation*}
S_n\phi=\sum_{k=0}^{n-1}\phi\circ T^k\quad\mbox{or}\quad S_t\phi=\int_0^t\phi\circ Y_s\dd s.
\end{equation*}
When the measure is a probability, the study of these quantities have been 
intensively studied  in the last half a century, with an always increasing interest.
%where $f$ is an observable of $A$ (or of $\mathcal M$). 
A first question is the law of large number (LLN), that is the almost sure convergence
of $\left({S_t\phi}/t \right)_{t>0}$ to the integral $I(\phi)$ of $\phi$ as $t\rightarrow +\infty$, which
happens to be true for any integrable function $f$ as soon as the system is ergodic (due to the Birkhoff-Khinchin theorem).
A second natural question is the establishment of central limit theorems (CLT), i.e.
of the convergence in distribution of $\left( S_t\phi/\sqrt{t}\right)_t$ ( as $t\rightarrow +\infty$) to a Gaussian random variable for centered square integrable
observables $\phi$, or, even more, the functional CLT (FCLT) that is the convergence of the family of processes $\left( (S_{st}\phi/\sqrt{t})_s\right)_t$ to a Brownian motion $(B_s)_s$, as $t\rightarrow +\infty$. In practice, CLT and FCLT hold true for smooth observables when the system is chaotic enough (satisfying nice mixing properties, see \cite{Ratner:1973,GuivarchHardy:1988,HennionHerve:2001}, etc.).\\ \smallskip

When the measure is infinite, it is natural to adress analogous questions, but the results are of different nature (we refer to \cite{Aaronson:1997} for a general reference on dynamical systems preserving an infinite measure). The first analogue of the LLN is given by the Hopf theorem, which states the almost everywhere convergence of $\left(S_t\phi/S_t\psi\right)_n$ to the ratio $I(\phi)/I(\psi)$ of the integrals of $\phi$ and $\psi$, for all couples of integrable functions $(\phi,\psi)$ with $\psi\not\equiv 0$, as soon as the system is conservative ergodic.
A second analogue of the LLN is the convergence in distribution in the strong sense of  $\left( S_t \phi/\mathfrak a_t\right)_t$ to a random variable (convergence in distribution in the strong sense means convergence in distribution with respect to any probability measure absolutely continuous with respect to the invariant measure). Note that, due to the Hopf theorem, it is enough to prove this result for a specific function $\phi\not\equiv 0$ to extend it to any integrable function. This second kind of analogue of LLN requires additional assumptions on the dynamical system.
Analogues of CLT for dynamical systems preserving an infinite invariant measure are non-degenerate limit theorems for $(S_t\phi)_t$ for null integral observables $f$. A classical analogue to the CLT in this context consists in establishing the convergence in distribution of $\left( S_t \phi/\mathfrak a_t^{\frac 12}\right)_t$ to some
random variable, with $\mathfrak a_t$ as in the above second analogue of the LLN.\\ \smallskip

The case of dynamical systems that can be represented by a $\mathbb Z^d$-extension over a probability preserving dynamical system is of particular interest.
As mentionned in
 \cite{Thomine:2013,Thomine:2014,Thomine:2015,PeneThomine:2019}, in this specific context, the question of the behaviour of Birkhoff sums is related
to the study of occupation times of $d$-dimensional random walks or Markov chains (see \cite{Dobrushin:1955,Kasahara:1981,Kasahara:1985,Kesten:1962}). Indeed, in the case of a transformation $T$, when the observable $\phi$ depends only on the  $\mathbb Z^d$-label in the $\mathbb Z^d$-extension, the ergodic sum $S_n\phi$ is the exact analogue of additive functionals of $d$-dimensional random walks or Markov chains.
Outside the cases of random walks or Markov chains, first results have been obtained by the second-named author  \cite{Thomine:2013,Thomine:2014,Thomine:2015} for Pomeau-Manneville maps, for $\mathbb Z^d$-extensions of Gibbs-Markov maps, for geodesic flows on a $\mathbb Z^d$-cover of a compact Riemannian variety with negative curvature. In \cite{PeneThomine:2019}, we established CLT in a general context of $\mathbb Z^d$-extensions of a dynamical system with nice spectral properties, including the $\mathbb Z^2$-periodic billiard model, but for observables depending
only on the $\mathbb Z^d$-label.\\ \smallskip

The aim of the present paper is to study ergodic sums of H\"older observables
of  the $\mathbb Z^2$-periodic Sinai billiard and of the
$\mathbb Z^2$-periodic Lorentz process, both with finite horizon. A first step in this direction is the property of conservativity and ergodicity which comes from \cite{JPC,
Simanyi:1989}
(thanks to \cite{Sinai:1970,BunimovichSinai:1981,BunimovichChernovSinai:1991}) and
which, combined with the Hopf theorem, ensures the above mentioned first analogue to the LLN. A second step is the proof 
by Dolgopyat, Sz\'asz and Varj\'u in
\cite{DolgopyatSzaszVarju:2008} of the above mentioned second analogue to the LLN, that in this context is
\begin{equation}\label{LFGNbill}
\forall \phi\in L^1,\quad  \frac{S_t \phi}{\ln t}   \Longrightarrow     I(\phi)\,  \mathcal E\, ,\quad\mbox{as }t\rightarrow +\infty\, ,
\end{equation}
where $\mathcal E$ is an exponential random variable and where $\Longrightarrow$
means the convergence in distribution in the strong sense.
A third step in this direction is the CLT with a normalization in $\sqrt{\ln t}$ obtained in \cite{PeneThomine:2019} for the billiard map and for observables depending only on the $\mathbb Z^d$-level.
In the present paper, our main result is a CLT and even a FCLT 
for H\"older observables $\phi$  (with null expectation) of the $\mathbb Z^2$-periodic Sinai billiard and of the
$\mathbb Z^2$-periodic Lorentz process of the following form
\begin{equation*}
\frac{{S}_{t} \phi}{\sqrt{\ln (t)}}
 \Longrightarrow \widetilde{\sigma}_\phi
\sqrt{\Ecal}\, \Ncal\, ,
\quad\mbox{as }t\rightarrow +\infty\, ,
\end{equation*}
with $\mathcal E$ as in \eqref{LFGNbill} and with $\Ncal$ a standard gaussian random variable independent of $\Ecal$, where $\widetilde\sigma_\phi$ is given by a Green-Kubo formula.
The above convergence result holds true providing $\phi$ satisfies some decay property at infinity. So it holds true at least for compactly supported H\"older functions with null integral.
More precisely, under the same assumptions and for any integrable function $\psi$, we prove the following joint FCLT
\begin{equation*}
 \left(\frac{S_{ts} \psi}{\ln (t)},\frac{S_{ts} \phi}{\sqrt{\ln (t)}}\right)_{s>0}
 \Longrightarrow \left(I(\psi)\, \Ecal, \, \widetilde{\sigma}_\phi \sqrt{\Ecal}\, \Ncal\right)_{s>0}\, ,
\end{equation*}
with the same notations as above. Roughly speaking, this means that, in distribution, 
\begin{equation*}
S_t\phi\approx \ln t \, I(\phi\, )\mathcal E+\sqrt{\ln t\, \mathcal E}\, \tilde\sigma_{\phi-I(\phi)\phi_0}\, \mathcal N+o(\sqrt{\ln t})\, ,\quad\mbox{as}\quad t\rightarrow +\infty\, ,
\end{equation*}
for $\phi,\phi_0$ two H\"older observables decaying quickly enough at infinity, with $I(\phi_0)=1$.
Note that, contrarily to the case of the classical FCLT, the limit we obtain is a process constant in time.
To prove our results, we use two methods producing different formulas for the "asymptotic variance" $\widetilde\sigma_\phi^2$ appearing in the CLT.\\ \smallskip
First, using the method of \cite{PeneThomine:2019}, we establish a general FCLT for $\mathbb Z^2$-extensions over a dynamical system satisfying general nice spectral properties (namely such that the step function satisfies a spectral local limit theorem). The fact that we restrict our study to $\mathbb Z^2$-extension with square integrable step functions (satisfying a classical limit theorem)
simplifies greatly the proof, makes its ideas appear much clearer than in \cite{PeneThomine:2019} and allows the generalization to H\"older functions, without adding technical complications.\\ \smallskip
Second, applying the method of \cite{Thomine:2013,Thomine:2014,Thomine:2015}, we obtain another way, based on induction, to prove the CLT for H\"older observables of the $\mathbb Z^2$-periodic billiard, under a slightly weaker assumption.\\ \medskip

The article is organized as follows. In Section~\ref{sec:main}, we present our context and results. We start by introducing in Section~\ref{subsec:gene} our general context of $\mathbb Z^2$-extensions of dynamical systems (in discrete time as well as in continuous time). In Section~\ref{subsec:bill}, we introduce the $\mathbb Z^2$-periodic Lorentz gas  (in discrete time as well as in continuous time). The rest of Section~\ref{sec:main} is devoted to the exposure of our main results, with a discussion on our technical assumptions.
In Sections~\ref{sec:operators} and \ref{sec:flow}, we prove our FCLT by the first method for
dynamical systems (first in the case of transformations in Section~\ref{sec:operators} and then in the case of flows in Section~\ref{sec:flow}).
In Section~\ref{sec:induc}, we prove the CLT via the second method (using induction).

\section{Context and main results}
\label{sec:main}

\subsection{General context}
\label{subsec:gene}

Given a probability preserving dynamical system $(A, \mu, T)$ and a function 
$F: A\to\Z^2$, we consider the infinite measure preserving 
dynamical system $(\widetilde{A}, \widetilde{\mu}, \widetilde{T})$ given 
by the $\Z^2$-extension of $(A, \mu, T)$ with step function $F$, 
i.e.\ $\widetilde{A}:=A\times\Z^2$, $\widetilde{\mu}=\mu\otimes \mathfrak m$, 
where $\mathfrak m$ is the counting measure on $\Z^2$ and 
$\widetilde{T}(x,a)=(T(x),a+F(x))$. Then, 
for all $(x, a)$ in $\widetilde{A}$ and $n \geq 0$,
\begin{equation*}
 \widetilde{T}^n(x,a) 
 = (T^n(x),a+S_n F(x))\, ,
\end{equation*}
where $S_nF$ is the ergodic sum:
\begin{equation*}
 S_n F
 :=\sum_{k=0}^{n-1}F\circ T^k \, .
\end{equation*}
We are interested in the asymptotic behaviour of the ergodic sums
\begin{equation*}
 \widetilde{S}_n f
 := \sum_{k=0}^{n-1}f\circ\widetilde{T}^k \, ,
\end{equation*}
for observables $f:\widetilde{A} \to \R$ in the particular case when 
$\int_{\widetilde{A}}f \dd \widetilde{\mu} = 0$. 

\smallskip

In the context of this article, the system $(A, \mu, T)$ 
shall be chaotic in a strong sense. More precisely, 
we shall assume that $(S_n F)_n$ satisfies a standard central limit theorem and, even more,  
a \textit{spectral local limit theorem} (see Assumption~\eqref{hyp:0} below), 
which is a strengthening of the more classical local limit theorem:
\begin{equation*}
 \mu(S_n F=a)
 \sim \frac{\Phi(a/\sqrt n)}{n}
\end{equation*}
for all $a \in \Z^2$, where $\Phi$ is the density of the Gaussian 
that is the limit distribution of $(S_n F/\sqrt{n})_n$. 
By Lemma~\ref{lem:0}, Assumption~\eqref{hyp:0} holds when the transfer operator $P$ 
of $T$,
dual to the Koopman operator, acts nicely on a Banach space  $\Bcal$ of 
integrable functions
or distributions.
This assumption is, in particular, satisfied 
by the collision map for Sinai billiards.

\smallskip

We shall also consider continuous-time versions of this problem, in two ways. 
The first way consists in defining the ergodic sums $\widetilde{S}_t f$ 
for real $t>0$ by linearization:
\begin{align*}
 \widetilde{S}_t f
 & :=(\lfloor t\rfloor+1 -t)S_{\lfloor t\rfloor} f +(t-\lfloor t\rfloor)S_{\lfloor t\rfloor +1} f \\
 & = S_{\lfloor t\rfloor} f +(t-\lfloor t\rfloor)f\circ\widetilde{T}^{\lfloor t\rfloor},
\end{align*}
which can be used to state functional limit theorems.

\smallskip

The second way consists in working directly with a continuous-time system. 
Given a measurable function $\tau: A \to (0,+\infty)$, 
the suspension flow $(\widetilde{\Mcal},\widetilde{\nu},(\widetilde{Y}_t)_t)$ of $(\widetilde{A},\widetilde\mu,\widetilde{T})$ 
with roof function $(x,a) \mapsto \tau(x)$ is the system:
\begin{equation*}
 \left\{
 \begin{array}{rcl}
 \widetilde{\Mcal} & := & \{(x,a,s) \in A\times\Z^2\times(0,+\infty)\, :\, s\in(0,\tau(x))\} \, , \\
 \widetilde{\nu} & := & (\widetilde{\mu}\otimes \Leb)_{|\widetilde{\Mcal}}\,  ,\\
 \widetilde{Y}_t(x,a,s) & := & \left(\widetilde{T}^{n_{t+s}(x)}(x,a),s+t-S_{n_{t+s}(x)}\tau(x)\right)\, ,
 \end{array}
 \right.
\end{equation*}
where $\Leb$ is the Lebesgue measure on $(0,+\infty)$ and 
$n_u(x):=\max\{n\geq 0\, :\, S_n\tau(x)\leq u\}$ for every $u\ge 0$. In this case, 
we define:
\begin{equation*}
 \widetilde{S}_t f
 := \int_0^t f \circ \widetilde{Y}_s \dd s\, .
\end{equation*}

\subsection{$\Z^2$-periodic Lorentz gas}
\label{subsec:bill}

We consider the displacement of a particle moving at unit speed in $\R^2$ with elastic reflection 
on a $\Z^2$-periodic configuration of dispersing obstacles, in finite horizon.

\smallskip

More precisely the billiard domain is given by $\R^2 \setminus\bigcup_{a\in\Z^2} \bigcup_{i=1}^I(O_i+a)$,
with obstacles $\{O_i+a\, ;\, i=1,\ldots,I,\ a\in\Z^2\}$ for some $I\geq 2$.
We assume that $(O_i)_{1 \leq i \leq I}$ is a finite family of precompact open convex sets in $\R^2$, 
whose boundaries are $\Ccal^3$ with non vanishing curvature. We assume that the closure
of the sets $O_i+a$ are pairwise disjoint. We assume moreover that 
$Q$ contains no line (finite horizon assumption).

\smallskip

We are interested in the behaviour of a point particle moving in $Q$ at unit speed, going straight inside $Q$ 
and obeying the Descartes reflection law at reflection times off
$\partial Q=\bigcup_{a\in\mathbb Z^2}\bigcup_{i=1}^I(\partial O_i+a)$.

\smallskip

A configuration is a couple position-speed $(q,\vec v)\in Q\times \Sbb^1$.
To avoid ambiguity, we allow only post-collisional vectors at reflection times, so that the configuration space is
\begin{equation*}
 \widetilde \Mcal
 :=\{(q,\vec v)\in Q\times \Sbb^1 :\, q\in\partial Q\ \Rightarrow\ \langle \vec n_q,\vec v\rangle\ge 0\}\, ,
\end{equation*}
where $\vec n_q$ denotes the unit vector normal to $\partial Q$ at $q$
oriented into $Q$.

\smallskip

The \textbf{Lorentz 
process} is the flow $(\widetilde Y_t)_t$ on $\widetilde \Mcal$ such that $\widetilde Y_t(q,\vec v)=(q_t,\vec v_t)$ 
is the configuration at time $t$ of a point particle that has configuration $(q,\vec v)$ at time $0$. 
This flow preserves the restriction $\tilde\nu$ on $\widetilde \Mcal$ of the Lebesgue measure $\Leb$ on $\R^2\times\Sbb^1$, 
normalized so that:
\begin{equation*}
 \Leb([0,1[^2\times\Sbb^1)
 = \frac{\pi}{\sum_{i=1}^I|\partial O_i|}\, .
\end{equation*}
This normalization will allow us to identify canonically this flow with a $\Z^2$-periodic suspension 
flow over a $\Z^2$-extension of a chaotic probability preserving
dynamical system, as described in Subsection~\eqref{subsec:gene}.

\smallskip

The dynamics at reflection times is the \textbf{$\Z^2$-periodic 
billiard system} $\left(\widetilde A,\widetilde\mu,\widetilde T\right)$, 
that is the first return map of the flow $\widetilde Y$ to the Poincar\'e section 
$\partial Q\times\Sbb^1$. Let $\widetilde A:=\{(q,\vec v)\in\widetilde\Mcal :\, q\in\partial Q\}$ 
be the set of configurations of post-collisional vectors off $\partial Q$. 
The map $\widetilde T:\widetilde A\rightarrow\widetilde A$ is the billiard transformation 
mapping a post-collisional configuration to the next post-collisional configuration. 
The measure $\widetilde\mu$ is given by:
\begin{equation}
\label{eq:ExpressionMuTilde}
 \dd\widetilde\mu(q,\vec v)
 = \frac{\cos\varphi}{2\sum_{i=1}^I|\partial O_i|} \dd r\dd\varphi\, ,
\end{equation}
where $r$ is the curvilinear absciss of $q$ on $\partial Q$,
and $\varphi$ is the angular measure in $[-\pi/2,\pi/2]$ of the angle
$\widehat{\left( \vec n_q,\vec v\right)}$.

This $\Z^2$-periodic billiard system $\left(\widetilde A,\widetilde\mu,\widetilde T\right)$
is a $\Z^2$-extension of the corresponding \textbf{Sinai billiard system} 
$(A,\mu, T)$. This Sinai billiard is the quotient of 
$\left(\widetilde A,\widetilde\mu,\widetilde T\right)$ modulo the action of $\Z^2$ 
on the position. 

\includegraphics[trim = 10mm 20mm 10mm 20mm, clip, width=14cm]{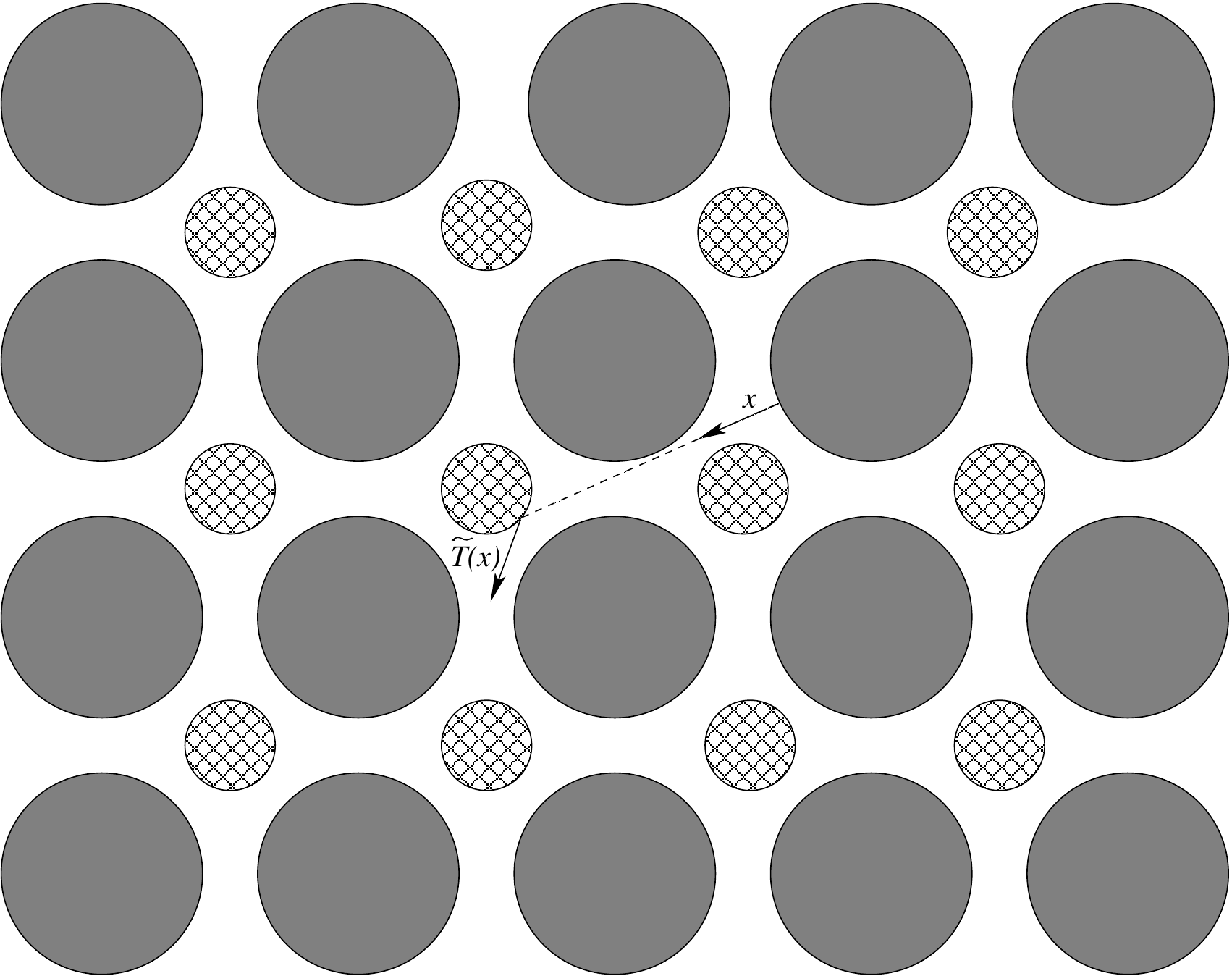}
\includegraphics[trim = 87mm 48.5mm 51mm 51mm, clip, width=4cm]{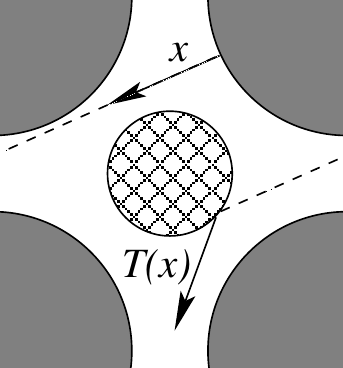}

More explicitely, the configuration set $A$ is the image of $\widetilde A$ 
by $\mathfrak p:\R^2\times\Sbb^1\rightarrow \Tbb^2\times\Sbb^1$ given by
$\mathfrak p(q,\vec v)=(\bar q,\vec v)$ where $\bar q=q+\Z^2\in\Tbb^2:=\R^2/\Z^2$.
By $\Z^2$-periodicity of $Q$, there exists a map $T:A\rightarrow A$ such that $T\circ\mathfrak p=\mathfrak p\circ \widetilde T$, 
which is the billiard map in the domain $\overline Q\subset \Tbb^2$
image of $Q$ by the canonical projection $\R^2\rightarrow\Tbb^2$.
The measure $\mu$ has the same expression as in Equation~\eqref{eq:ExpressionMuTilde}.

\smallskip

The function $F : A \to \Z^2$ giving the size of the jumps is defined by $F (q,\vec v) = b-a$ 
whenever 
$\widetilde T (q,\vec v,a) \in \bigsqcup_{i=1}^I (\partial O_i+b) \times \mathbb S^1$;
the $\Z^2$-periodicity of the billiard 
table ensures that this function is well-defined.

\smallskip

%This map $T$ preserves the probability measure
%$\mu$ such that $\dd\mu(\bar q,\vec v)=\cos\varphi\, \dd r\dd\varphi/(2 \sum_{i=1}^I|\partial O_i|)$, where $r$ is the curvilinear absciss of $\bar q$ on $\partial \overline Q$,
%where $\varphi$ is the angular measure in $[-\pi/2,\pi/2]$ of the angle
%$\widehat{\left( \vec n_q,\vec v\right)}$ and where $|\partial O_i|$ stands for the length of $\partial O_i$. 

Let $\tau : A \to \R_+^*$ be the free path length of a particle on $\overline Q$ :
\begin{equation*}
 \tau(q,\vec v)
 = \min\{s>0 :\, q+s\vec v\in\partial 
\overline Q
\}\, .
\end{equation*}
By $\Z^2$-periodicity of $Q$, the function $\widetilde{\tau}: \, (q,\vec v, a) \mapsto \tau(q,\vec v)$ 
defined on $\widetilde A$ is the free path length of a particle on $Q$. 
The Lorentz gas $\left(\widetilde\Mcal,\tilde\nu,(Y_t)_t\right)$ is then canonically 
identified with the suspension flow over $\left(\widetilde A,\widetilde\mu,\widetilde T\right)$ 
with roof function $\widetilde{\tau}$.

%this identification being 
%$((\bar q,\vec v),a,s)$ with $\widetilde Y_s(q+a,\vec v)$ for every
%$((\bar q,\vec v),a,s)\in A\times\Z^2\times[0,\infty)$, with $s<\tau(q,\vec v)$, where $q\in p_{\Z^2}^{-1}(\{\bar q\})\cap\bigcup_{i=1}^I\partial O_i$ and for every .
%Observe that $\tau$ can be quotiented by $\mathfrak p$ since, by $\Z^2$-periodicity of $Q$, $\tau(q+a,\vec v)=\tau(q,\vec v)$ for all $(q,\vec v)\in A$ and all $a\in\Z^2$.

\subsection{Results for transformations}

We state our main limit theorem under abstract conditions; our other results -- 
applications to billiards or to continuous-time systems -- will follow from that.

\begin{theorem}
\label{thm:gene}

Assume that $(\widetilde{A}, \widetilde{\mu}, \widetilde{T})$ is conservative and ergodic.
%Let $(\widetilde{A}, \widetilde{\mu}, \widetilde{T})$ be a $\Z^2$-extension of $(A, \mu, T)$ with step function $F$.
Let $\eta>0$ and $p$, $p^* \in [1, \infty]$ such that $p^{-1}+(p^*)^{-1}=1$. 
Let $(\Bcal,\norm{\cdot}{\Bcal})$ be a Banach space (of functions or distributions) containing $\mathbf{1}:=\mathbf{1}_A$ and such that
\begin{itemize}
\item either $p=1$ (so $p^*=\infty$) and $\Ebb_\mu[\cdot]$ is continuously extended from $\Bcal\cap L^1(A,\mu)$ to $\Bcal$;
\item or $p>1$ and $\Bcal \hookrightarrow \Lbb^p (A, \mu)$ (where $\hookrightarrow$ is a continous injection).
\end{itemize}
Assume moreover the following spectral local limit condition:
\begin{equation}
\label{hyp:0}
\sup_{a\in\Z^d,\ h\in\Bcal\, :\, \norm{h}{\Bcal} \leq 1} \norm{P^\ell \left(\mathbf{1}_{\{S_\ell F=a\}}\, h\right) -\frac{\Phi\left(\frac{a}{\sqrt{\ell}}\right)}{\ell}\Ebb_\mu[h] }{\Bcal} 
= O(\ell^{-1-\eta})\, ,
\end{equation}
where $\Phi$ is a two-dimensional non-degenerate Gaussian density function.
%with $Z$ a standard $d$-dimensional Gaussian random variable and $\Sigma^2_F :=\int_{\widetilde{A}} f^{\otimes 2} \dd \widetilde{\mu} + \sum_{k \geq 1} \int_{\widetilde{A}} (f \otimes (f \circ {T}^k) +(f\circ T^k) \otimes f ) \dd \widetilde{\mu}$.
%Let $(\ln (n)):=\sum_{k=1}^n \mu(S_nF=0)$. 

\smallskip

Let $f$, $g:\widetilde{A} \to \R$ be such that:
\begin{itemize}
 \item $\sum_{a\in\Z^2} (1+|a|^{\varkappa}) \norm{ f(\cdot,a)\times\cdot}{\Lcal(\Bcal,\Bcal)} <+ \infty$ for some $\varkappa >0$;
 \item $\sum_{a\in\Z^2} \norm{f(\cdot,a)}{\Lbb^{p^*} (A, \mu)} < +\infty$;
 \item $\int_{\widetilde{A}} f \dd\widetilde{\mu} = 0$;
 \item $g \in \Lbb^1 (\widetilde{A}, \widetilde{\mu})$.
\end{itemize}
Then the following sum over $k$ is absolutely convergent:
\begin{equation}
\label{AAA2}
\tilde \sigma^2 (f)
:= \int_{\widetilde{A}} f^2 \dd \widetilde{\mu} + 2\sum_{k \geq 1} \int_{\widetilde{A}} f \cdot f \circ \widetilde{T}^k \dd \widetilde{\mu}.
\end{equation}
Moreover, for every $0 < T_1 < T_2$, as $n\to +\infty$,
\begin{equation}\label{eq:ConvergenceDistributionGene}
\left(\frac{\widetilde{S}_{nt} g}{\ln (n)},\frac{\widetilde{S}_{nt} f}{\sqrt{\ln (n)}}\right)_{t\in[T_1,T_2]}
\longrightarrow \left(\int_{\widetilde{A}} g\dd\widetilde{\mu}\, \Phi(0)\Ecal, \, \widetilde \sigma (f) \sqrt{\Phi(0)\Ecal}\, \Ncal\right)_{t\in[T_1,T_2]}
\end{equation}
in distribution in $C([T_1,T_2])$, with respect to any probability measure 
absolutely continuous\footnote{The property of convergence in distribution with respect to any absolutely continuous probability measure 
is sometimes called \textit{strong convergence in distribution}~\cite{Aaronson:1997}.} with respect to $\widetilde{\mu}$, and where $\Ecal$ and $\Ncal$ 
are two independent random variables, with respectively standard exponential 
and standard Gaussian distributions.
\end{theorem}

Note that, since $\mathbf{1} \in \Bcal$, the condition $\sum_{a\in\Z^2} (1+|a|^{\varkappa}) \norm{ f(\cdot,a)\times}{\Lcal(\Bcal,\Bcal)} <+ \infty$ 
implies that $\sum_{a\in\Z^2} (1+|a|^{\varkappa}) \norm{ f(\cdot,a)}{\Bcal} <+ \infty$. 
Note also that the random variable $\sqrt{\Ecal} \Ncal$ in Equation~\eqref{eq:ConvergenceDistributionGene} follows 
a standard (centered with variance $1$) Laplace distribution.

\smallskip

The condition $\sum_{a\in\Z^2} \norm{f(\cdot,a)}{\Lbb^{p^*} (A, \mu)} < +\infty$ in Theorem~\ref{thm:gene} 
is only used in the proof of the functional convergence, and is not necessary for the convergence in distribution 
of $\left(\frac{\widetilde{S}_n g}{\ln (n)},\frac{\widetilde{S}_n f}{\sqrt{\ln (n)}}\right)_{n \geq 0}$.

\smallskip

We shall prove Theorem~\ref{thm:gene} using the method of moments. The 
same strategy was used in~\cite{PeneThomine:2019}. However, our setting provides some welcome simplifications, 
allowing us to apply our method to more general observables $f$ than the ones considered in~\cite{PeneThomine:2019}.  
These simplifications come namely from the summability in $\ell$ in Equation~\eqref{hyp:0}, 
as well as the summability of other error terms.

\smallskip

As proved in Lemma~\ref{lem:0}, the hypothesis~\eqref{hyp:0} is satisfied
under quite general spectral assumptions, which are stated in Hypothesis~\ref{hyp:HHH}.
In particular, Hypothesis~\ref{hyp:HHH} holds for the collision map 
associated with Sinai billiards~\cite{DemersPeneZhang}, from which we deduce:

\begin{corollary}
\label{cor:billtransfo}

Let $(\widetilde{A},\widetilde{\mu},\tilde F)$ be the $\Z^2$-periodic billiard system
presented in Subsection~\ref{subsec:bill}. Let $f$, $g:\widetilde{A}\to \R$ be such that:
\begin{itemize}
 \item $f$ is $\eta$-H\"older
 on the continuity domain of $\widetilde T$
%{\color{blue}sets on which $\widetilde T,...,\widetilde T^\ell$ are continuous,
 for some $\eta >0$;
%and some $\ell\ge 0$;}
 \item $\sum_{a\in\Z^2} |a|^\varkappa \norm{f(\cdot,a)}{\eta}<+\infty$ for some $\varkappa >0$;
 \item $\int_{\widetilde{A}} f \dd\widetilde{\mu} = 0$;
 \item $g \in \Lbb^1 (\widetilde{A},\widetilde{\mu})$,
\end{itemize}
where 
$\norm{\cdot}{\eta}$ is the maximal $\eta$-H\"older norm with respect to the
euclidean metric on $M$ on the continuity domains of $T$.
\smallskip

Then
\begin{equation*}
 \left(\frac{\widetilde{S}_n g}{\ln (n)},\frac{\widetilde{S}_n f}{\sqrt{\ln (n)}}\right)
 \longrightarrow \left(\frac{ \int_{\widetilde{A}} g\dd\widetilde{\mu} }{ 2\pi\sqrt{\det(\Sigma^2)} } \, \Ecal, \, \frac{\widetilde \sigma (f)}{\sqrt{2 \pi} (\det(\Sigma^2))^{\frac{1}{4}}} \sqrt{\Ecal}\, \Ncal\right)\, ,
\end{equation*}
in distribution with respect to any probability measure absolutely continuous 
with respect to $\widetilde{\mu}$, and where $\Ecal$ and $\Ncal$ are two independent random 
variables, with respectively standard exponential and standard Gaussian distributions, 
with $\tilde\sigma^2 (f)$ given by Equation~\eqref{AAA2} 
and with $\Sigma$ the invertible symmetric positive matrix such that
$\Sigma^2 = \sum_{k\in\Z}\Ebb\left[F\otimes(F\circ T^k)\right]$.
\end{corollary}

\begin{proof}
This corollary is a direct consequence of Theorem~\ref{thm:gene}
and Lemma~\ref{lem:0} thanks to \cite[Theorem 3.17]{DemersPeneZhang} 
(ensuring Hypothesis \ref{hyp:HHH} with $p=1$) and to~\cite[Lemma 5.3]{DemersZhang:2014} 
or \cite[Lemma 4.5]{DemersPeneZhang} (ensuring that
$\norm{f(\cdot,a)\times}{\Lcal(\Bcal,\Bcal)}\leq C\norm{f(\cdot,a)}{\eta}$).
\end{proof}

While the results above are proved using the method of moments, we shall also 
show how to prove similar propositions using induced systems. The strategy follows closely 
that of~\cite[Proposition~6.12]{Thomine:2015}, which was in the setting 
of geodesic flows in negative curvature, with some improvements from~\cite{PeneThomine:2019}. 
The main difference in the present article is that, in the context of Sinai billiards, we 
use Young tower in order to introduce a symbolic coding of the trajectories.

\begin{proposition}
\label{prop:BillardParInduit}

Let $(\widetilde{A},\widetilde{\mu},\widetilde T)$ be the $\Z^2$-periodic billiard system
presented in Subsection~\ref{subsec:bill}.
Let $f:\widetilde{A} \to \R$ be such that:
\begin{itemize}
 \item $f$ is $\eta$-H\"older on the continuity domains of $\widetilde T$, with a uniformly bounded $\eta$-H\"older norm, 
 for some $\eta >0$;
 \item $\sum_{a\in\Z^2}(1+\ln_+ |a|)^{\frac{1}{2}+\varkappa} \norm{f(\cdot,a)}{\infty}<+\infty$ for some $\varkappa >0$;
 \item $\int_{\widetilde{A}} f \dd\widetilde{\mu} = 0$.
\end{itemize}

\smallskip

Then
\begin{equation*}
 \frac{\widetilde{S}_n f}{\sqrt{\ln (n)}} 
 \longrightarrow \frac{\widehat{\sigma} (f)}{\sqrt{2\pi}(\det(\Sigma^2))^{\frac{1}{4}}} L \, ,
\end{equation*}
in distribution with respect to any probability measure absolutely continuous 
with respect to $\widetilde{\mu}$, where $L$ follows a centered Laplace distribution 
of variance $1$, with $\widehat{\sigma} (f)$ given by Equation~\eqref{eq:VarianceInduit} 
and with $\Sigma$ the invertible symmetric positive matrix such that
$\Sigma^2=\sum_{k\in\Z} \Ebb\left[F\otimes(F\circ T^k)\right]$.

\smallskip

In addition, $\widehat{\sigma} (f) = 0$ if and only if $f$ is a coboundary.
\end{proposition}

Comparing the conclusions of Corollary~\ref{cor:billtransfo} and Proposition~\ref{prop:BillardParInduit}, 
one has $\widetilde{\sigma} (f) = \widehat{\sigma} (f)$ whenever $f$ satisfies the assumptions of 
Corollary~\ref{cor:billtransfo}. This equality has a deep dynamical consequences~\cite{PeneThomine:2019}.

\smallskip

The assumptions of Proposition~\ref{prop:BillardParInduit} are slightly weaker than those 
of Corollary~\ref{cor:billtransfo}. The conclusions of the former are also weaker, dealing with the 
limit distribution of $S_n f$ and not the limit joint distribution of $(S_n g, S_n f)$. 
The stronger result should hold under the assumptions of Proposition~\ref{prop:BillardParInduit}, 
but one would start from~\cite[Theorem~1.7]{Thomine:2013}, which is beyond the scope of this article.
On the other hand, it should also be possible to weaken the assumptions of Corollary~\ref{cor:billtransfo} (dynamically H\"older observables satisfying a weaker decay condition expressed only in terms of supremum norm), with a less direct and more technical proof (using approximations).

\subsection{Results for flows}

Theorem~\ref{thm:gene} admits a version for semiflows:

\begin{theorem}
\label{thm:geneflot}

Assume that $(\widetilde{A}, \widetilde{\mu}, \widetilde{T})$ is conservative and ergodic.
%Let $(\widetilde{A}, \widetilde{\mu}, \widetilde{T})$ be a $\Z^2$-extension of $(A, \mu, T)$ with step function $F$.
Let $\eta>0$ and $p$, $p^* \in [1, \infty]$ such that $p^{-1}+(p^*)^{-1}=1$.
Let $(\Bcal,\norm{\cdot}{\Bcal})$ be a Banach space satisfying the assumptions of Theorem~\ref{thm:gene}.

\smallskip

Let $\phi$, $\psi: \widetilde{\Mcal} \to \R$ be such that:
\begin{itemize}
 \item $\sum_{a\in\Z^2} (1+|a|^{\varkappa}) \norm{ G(\phi)(\cdot,a)\times\cdot}{\Lcal(\Bcal,\Bcal)} <+ \infty$ for some $\varkappa >0$;
 \item $\sum_{a\in\Z^2} \norm{G(|\phi|)(\cdot,a)}{\Lbb^{p^*} (A, \mu)} < +\infty$;
 \item $\int_{\widetilde{M}} \phi\dd\widetilde{\mu} = 0$;
 \item $\psi \in \Lbb^1 (\widetilde{\Mcal}, \widetilde{\nu})$.
\end{itemize}
Then, for every $0 < s_1 < s_2$, as $t \to \infty$, 
\begin{equation}\label{eq:ConvergenceDistributionGeneFlot}
\left(\frac{\widetilde{S}_{ts} \psi}{\ln (t)},\frac{\widetilde{S}_{ts} \phi}{\sqrt{\ln (t)}}\right)_{s\in[s_1,s_2]}
\longrightarrow \left(\int_{\widetilde{\Mcal}}\psi\dd\widetilde{\nu}\, \Phi(0)\Ecal, \, \widetilde{\sigma} (G(\phi)) \sqrt{\Phi(0)\Ecal}\, \Ncal\right)_{s\in[s_1,s_2]}\, ,
\end{equation}
in distribution in $\Ccal ([s_1,s_2], \R)$ with respect to any probability measure 
absolutely continuous with respect to $\widetilde{\nu}$. In Equation~\eqref{eq:ConvergenceDistributionGeneFlot},
$\Ecal$ and $\Ncal$ are two independent random variables with respectively 
standard exponential and standard Gaussian distributions.
\end{theorem}

As Theorem~\ref{thm:gene} was applied to the collision map for Sinai billiards, 
so does Theorem~\ref{thm:geneflot} to the flow on Sinai billiards (i.e.\ to the two-dimensional Lorentz gas model). 
In order for the Lorentz gas to fit our setting, we see it as a suspension flow 
of height $\widetilde{\tau}$ over its collision map $(\widetilde{A},\widetilde{\mu},\widetilde{T})$, which 
was the object of Corollary~\ref{cor:billtransfo}. 

%The height $\tau$ of the suspension flow is the free flight length until the next reflection, and $\widetilde{\nu}$ is the Lebesgue measure on $\widetilde{\Mcal}$.

\begin{corollary}\label{coroLorentz}

Let $(\widetilde{\Mcal},\widetilde{\nu},(\widetilde{Y}_t)_t)$ be the $\Z^2$-periodic Lorentz gaz described above.
Let $\eta$, $\varkappa>0$, and denote by $\norm{\cdot}{\eta}$ the $\eta$-H\"older norm on $\widetilde{\Mcal}$.

\smallskip

Let $\phi$, $\psi: \widetilde{\Mcal} \to \R$ be such that:
\begin{itemize}
 \item $\sum_{a=(a_1,a_2)\in\Z^2}|a|^\varkappa \norm{\phi_{|[a_1,a_1+1]\times[a_2,a_2+1]}}{\eta} <+\infty$;
 \item $\int_{\widetilde{\Mcal}}\phi\dd\widetilde{\nu}=0$;
 \item $\psi \in \Lbb^1 (\widetilde{\Mcal}, \widetilde{\nu})$;
\end{itemize}
Then, for every $0<s_1<s_2$, as $t \to +\infty$,
\begin{equation}
 \left(\frac{\widetilde{S}_{ts} \psi}{\ln (t)},\frac{\widetilde{S}_{ts} \phi}{\sqrt{\ln (t)}}\right)_{s\in[S_1,S_2]}
 \longrightarrow \left(\frac{\int_{\widetilde{\Mcal}}\psi\dd\widetilde{\nu}}{2\pi\sqrt{\det(\Sigma^2)}} \Ecal, \, \frac{\widetilde{\sigma} (G(\phi))}{\sqrt{2 \pi} (\det(\Sigma^2))^{\frac{1}{4}}} \sqrt{\Ecal}\, \Ncal\right)_{s\in[s_1,s_2]}\, ,
\end{equation}
with 
\begin{equation*}
 \widetilde{\sigma} (G(\phi))^2
 := \sum_{k\in\Z} \int_{\widetilde{A}} G(\phi) \cdot G(\phi) \circ \widetilde{T}^k \dd \widetilde{\mu}\, .
\end{equation*}
\end{corollary}

\subsection{Spectral hypotheses}
\label{subsec:technical}

All the results above hold whenever Assumption~\eqref{hyp:0} is satisfied. To 
finish this introduction, we now relate this assumption to more classical spectral conditions 
on the transfert operator associated with $(A, \mu, T)$.

\smallskip

The transfer operator $P$ is defined, for $f \in \Lbb^1 (A, \mu)$, by:
\begin{equation*}
 \int_A P(f) \cdot g \dd \mu 
 = \int_A f \cdot g \circ T \dd \mu 
 \quad \forall g \in \Lbb^\infty (A, \mu)\, . 
\end{equation*}
Recall that $F : A \to \Z^2$. Let  $\Tbb^2 := \R^2/(2\pi\Z)^2$. 
We define a family of twisted transfer operators $(P_u)_{u \in \Tbb^2}$ by:
\begin{equation*}
 P_u(h) 
 := P (e^{i \langle u, F \rangle} h)
\end{equation*}
for all $h \in \Lbb^1 (A, \mu)$. Note that:
\begin{equation}\label{Puk}
P_u^k (h)
= P^k \left(e^{i \langle u, S_k F \rangle} h \right)\, .
\end{equation}
The idea to study the spectral properties of $P_u$ to establish limit theorems goes back 
to the seminal works by Nagaev~\cite{Nagaev:1957, Nagaev:1961} and Guivarc'h~\cite{GuivarchHardy:1988} 
and has been deeply generalized by Keller and Liverani in~\cite{KellerLiverani:1999}. 
We refer to the book by Hennion and Herv\'e~\cite{HennionHerve:2001} for an overview of the important results 
that can be proved by this method.

\smallskip

The more usual spectral conditions are: 

\begin{hypothesis}[Spectral hypotheses]
\label{hyp:HHH}

There exists a complex Banach space $(\Bcal,\norm{\cdot}{\Bcal})$ of functions 
or of distributions defined on $A$, on which $P$ acts continuously, and such that:
\begin{itemize}
\item $\mathbf{1} \in \Bcal$ and $\Ebb_\mu[\cdot]$ extends continuously from $\Bcal\cap L^1(A,\mu)$ to $\Bcal$;
\item for every $a\in\Z^2$, the multiplication by $f(\cdot,a)$ belongs to $\Lcal(\Bcal,\Bcal)$;
\item There exist an open ball $U \subset \Tbb^2$ containing $0$, two constants $C>0$ and $r \in(0,1)$, 
continuous functions $\lambda_\cdot:U \to \C$ and $\Pi_\cdot$, $R_\cdot:U \to \Lcal(\Bcal,\Bcal)$ 
such that
\begin{equation}\label{eq:decomp}
P_u^n 
=\lambda_u^n\Pi_u+R_u^n\,
\end{equation}
with:
\begin{align}
%\Pi_u R_u     & = R_u \Pi_u = 0,							     \label{eq:PiuRu} \\
%\Pi_u^{2} & = \Pi_{u}, 								     \label{eq:PiM+1} \\
\norm{\Pi_u - \Ebb_\mu [\cdot]}{\Lcal(\Bcal,\Bcal)} & \leq C |u|  \quad \forall u\in U, \label{eq:Pi0} \\
\sup_{u\in U} \norm{R_u^k}{\Lcal(\Bcal,\Bcal)} +\sup_{u\in \Tbb^2 \setminus U} \norm{P_u^k}{\Lcal(\Bcal,\Bcal)} & \leq C r^k,				\label{eq:majoRu}
\end{align}
\item there exists an invertible positive symmetric matrix $\Sigma$ and $\varepsilon >0$ such that, as $u \to 0$,
\begin{equation}\label{eq:lambdau}
 \lambda_u
 = e^{-\frac{\langle\Sigma^2 u,u\rangle}{2}} + O\left(|u|^{2+\varepsilon} \right)\, .
\end{equation}
\end{itemize}
\end{hypothesis}

\begin{lemma}
\label{lem:0}

Assume that the Hypotheses~\ref{hyp:HHH} hold. Let $\Phi(x)=\frac{e^{-\frac{\langle \Sigma^{-2}x,x \rangle}{2}}}{2\pi\sqrt{\det(\Sigma^2)}}$ 
and $\eta \in (0,\varepsilon/2]$. Then Equation~\eqref{hyp:0} holds:
\begin{equation*}
\sup_{\substack{a\in\Z^2,\ h\in\Bcal\,  \\ \norm{h}{\Bcal} \leq 1}} \norm{P^\ell \left(\mathbf{1}_{\{S_\ell F=a\}}\, h\right) -\frac{\Phi\left(\frac{a}{\sqrt{\ell}}\right)}{\ell}\Ebb_\mu[h] }{\Bcal} 
= O(\ell^{-1-\eta})\, .
\end{equation*}
\end{lemma}

\begin{proof}

Let $Q_{\ell,a}$ be the operator acting on any $h \in \Lbb^1 (A, \mu)$ by:
\begin{equation*}
Q_{\ell,a} (h) (x)
:= P^\ell \left(\mathbf{1}_{\{S_\ell F=a\}} h\right) (x) \, .
\end{equation*}
Due to Equation~\eqref{Puk},
\begin{equation}\label{Qlaprop}
Q_{\ell,a} (h) 
= \frac{1}{(2\pi)^2} \int_{\Tbb^2} e^{-i \langle u, a \rangle} P_u^\ell (h) \dd u\, ,
\end{equation}
and in particular $Q_{\ell,a}$ acts on $\Bcal$.
From Hypothesis~\ref{hyp:HHH}, and up to taking a smaller neighborhood $U$, 
there exist constants $C_0$, $c_0 > 0$ such that $\norm{P_u}{\Lcal (\Bcal, \Bcal)} \leq C_0$ and
\begin{equation*}
\max \left\{ |\lambda_u|, \left|e^{-\frac{\langle\Sigma^2 u, u\rangle}{2}} \right| \right\}
\leq e^{-c_0 |u|^2}
\end{equation*}
for all $u \in U$. Due to Equations~\eqref{Qlaprop} and~\eqref{eq:majoRu}, 
\begin{equation}\label{eqnumero0}
\sup_{a\in\Z^2} \norm{Q_{\ell,a}-\frac{1}{(2\pi)^2}\int_U e^{-i\langle u, a \rangle} \lambda_u^\ell \Pi_u \dd u}{\Lcal(\Bcal,\Bcal)} 
= O(r^\ell).
\end{equation}
In addition, there exists $C_1>0$ such that, for every $u \in U$, 
\begin{align}
\norm{\lambda_u^\ell \Pi_u -e^{-\ell\frac{\langle\Sigma^2 u, u\rangle}{2}} \Pi_0}{\Lcal(\Bcal,\Bcal)} 
& \leq |\lambda_u|^\ell \norm{\Pi_u-\Pi_0}{\Lcal (\Bcal,\Bcal)} + \left| \lambda_u^\ell- e^{-\ell\frac{\langle\Sigma^2 u, u\rangle}{2}} \right| \norm{\Pi_0}{\Lcal(\Bcal,\Bcal)} \nonumber \\
& \leq C_1 (|u|+\ell |u|^{2+\varepsilon})e^{-\ell \frac{c_0|u|^2}{2} }, 
\nonumber%\label{eqnumero1}
\end{align}
due to the asymptotic expansion of $u \mapsto \lambda_u$ and to Equation~\eqref{eq:Pi0}.
Hence, using the change of variable $u=v/\sqrt{\ell}$,
\begin{align}
\sup_{a\in\Z^2}\left\| \frac{1}{(2\pi)^2} \int_U e^{-i\langle u, a \rangle} \lambda_u^\ell \Pi_u \dd u  \right. & 
\left. - \frac{1}{(2\pi)^2}\int_U e^{-i\langle u, a \rangle} e^{-\ell\frac{\langle\Sigma^2 u, u\rangle}{2}} \Pi_0 \dd u \right\|_{\Lcal(\Bcal,\Bcal)} \nonumber \\
& \leq C_1 \int_U (|u|+\ell |u|^{2+\varepsilon})e^{-\ell\frac{c_0|u|^2}{2}} \dd u \nonumber\\
& \leq \frac{C_1}{\ell} \int_{\R^2} \left(\frac{|v|}{\sqrt{\ell}}+\ell \frac{|v|^{2+\varepsilon}}{\ell^{1+\frac{\varepsilon}{2}}} \right) e^{-\frac{c_0|v|^2}{2}} \dd v
 = O\left( \frac{1}{\ell^{1+\frac{\varepsilon}{2}}} \right). \label{eqnumero5}
\end{align}
Finally, using the same change of variable,
\begin{align}
\sup_{a\in\Z^2}\left| \frac{1}{(2\pi)^2} \int_U \right. & \left. e^{-i\langle u, a \rangle} e^{-\ell\frac{\langle\Sigma^2 u, u\rangle}{2}} \dd u - \frac {1}{\ell}\Phi\left(\frac{a}{\sqrt{\ell}}\right) \right| \nonumber\\
& = \sup_{a\in\Z^2}\left| \frac{1}{(2\pi)^2 \ell} \int_{\sqrt{\ell}\, U} e^{-i\frac{\langle v, a \rangle}{\sqrt{\ell}}} e^{-\frac{\langle\Sigma^2 v, v\rangle}{2}} \dd v- \frac {1}{(2\pi)^2\ell}\int_{\R^2} e^{-i\frac{\langle v, a \rangle}{\sqrt{\ell}}} e^{-\frac{\langle\Sigma^2 u, u\rangle}{2}} \dd v \right| \nonumber\\
& \leq \left| \frac{1}{(2\pi)^2 \ell} \int_{\R^2\setminus\sqrt{\ell}\, U} e^{-\frac{\langle\Sigma^2 v, v\rangle}{2}} \dd v\right| = O \left(\ell^{-2} \right)\, . \label{eqnumero6}
\end{align}
The lemma follows from Equations~\eqref{eqnumero0}, \eqref{eqnumero5} and~\eqref{eqnumero6}.
\end{proof}

\section{Proof of Theorem~\ref{thm:gene}}
\label{sec:operators}

This section is devoted to the proof of Theorem~\ref{thm:gene}. We proceed in two steps. First, we prove the 
convergence in distribution for $t=1$ and then we shall extend the convergence in distribution to a functional convergence.
The method we use here is close to the one used in~\cite{PeneThomine:2019}. In ~\cite{PeneThomine:2019}, we considered 
a wide family of dynamical systems ($\Z^d$-extensions with $d\in\{1,2\}$ and  $(S_nF)_n$ satisfying a standard or 
nonstandard central limit theorem involving a stable distribution), but we considered also a specific family of observables $f$ 
(which were assumed to be constant on each cell, i.e.\ satisfying $f(x,a)=f(y,a)$ for every $x,y\in A$). 
In the present paper, we focus on more specific dynamical systems (with $d=2$ and $(S_nF)_n$ satisfying a standard 
central limit theorem), which includes the Lorentz process. This more strigent context allows significative simplifications 
(due to summable error terms) which make much clearer the understanding of our argument and allow us to generalize 
the method used in~\cite{PeneThomine:2019} to more general observables.

\subsection{Convergence in distribution for $t=1$}

This section is devoted to the proof of Theorem~\ref{thm:gene} for $t=1$. 
In other words, under the hypotheses of Theorem~\ref{thm:gene}, we shall show that:
\begin{equation}\label{eq:TCL}
\left(\frac{\widetilde{S}_n g}{\ln (n)},\frac{\widetilde{S}_n f}{\sqrt{\ln (n)}}\right)
\longrightarrow \left(\int_{\widetilde{A}}g\dd\widetilde{\mu}\, \Phi(0)\Ecal, \, \widetilde \sigma (f) \sqrt{\Phi(0)\Ecal}\, \Ncal\right),\quad\mbox{as}\ n\to +\infty\, ,
\end{equation}
where the convergence is in distribution as $n \to + \infty$, with respect to any absolutely continuous probability measure.

\begin{proof}[Proof of Theorem~\ref{thm:gene} for $t=1$]
Since $\widetilde{T}$ is ergodic, due to Hopf's ergodic theorem~\cite[\S$14$, Individueller Ergodensatz f\"ur Abbildungen]{Hopf:1937}, 
we assume without any loss of generality that
$g(x,a) = \mathbf{1}_0 (a)$,
which shall significantly 
simplify the computations in the proof of Lemma~\ref{lem:Borne01}.

\smallskip

Set $\afrak_n=\ln (n)$, so that $\afrak_n\sim \sum_{k=1}^n k^{-1}$ as $n\to +\infty$. Due to \cite[Theorem~1]{Zweimuller:2007}, 
it is enough to prove the convergence in distribution with respect to $\widetilde{T}_*(\mu\otimes\delta_0)$, 
i.e.\ the convergence in distribution of $(\frac{\Zcal_n g}{\afrak_n},\frac{\Zcal_n f}{\sqrt{\afrak_ n}})_n$ 
with respect to $\mu$, where:
\begin{equation*}
\Zcal_n h(x):=(\widetilde{S}_n h)\circ\widetilde{T}(x,0)
= \sum_{k=1}^{n} h\left(T^kx,S_kF(x)\right).
\end{equation*}
The convergence in distribution of  $(\frac{\Zcal_n g}{\afrak_n},\frac{\Zcal_n f}{\sqrt{\afrak_n}})_n$ 
is equivalent to the convergence in distribution of 
$\alpha\frac{\Zcal_n g}{\afrak_n}+\beta\frac{\Zcal_n f}{\sqrt{\afrak_n}}$ 
for every $\alpha,\beta\in\R$. Let us fix $\alpha$, $\beta\in\R$ for the remainder of the proof.

\smallskip

We use the method of moments. Setting $h_n(x,a):=\frac\alpha{\afrak_n}g(x,a)+\frac\beta{\sqrt{\afrak_n}} f(x,a)$, 
due to Carleman's criterion~\cite[Chap.~XV.4]{Feller:1966}, it is enough to prove that, 
for all $m \geq 0$,
\begin{equation}\label{cvgcemoment}
 \lim_{n\to +\infty}\Ebb_\mu \left[(\Zcal_n h_n)^m\right]
 = \Ebb\left[(\alpha\Phi(0)\Ecal+\beta\sqrt{\Phi(0)\Ecal}\, \widetilde\sigma (f) \Ncal)^m\right]\, .
\end{equation}
Let us fix an integer $m \geq 0$ for the remainder of the proof. Then, for all $n$:
\begin{align*}
\Ebb_\mu \left[ (\Zcal_nh_n)^m \right] 
& = \Ebb_\mu \left[ \left(\sum_{k=1}^n (h_n (T^k (\cdot),S_k F (\cdot)) \right)^m \right] \\
& = \sum_{k_1,\ldots,k_m = 1}^n \sum_{d_1,\ldots,d_m \in \Z^2} \Ebb_\mu \left[ \prod_{s=1}^m h_n(T^{k_s} (\cdot),d_s) \mathbf{1}_{\{S_{k_s} F (\cdot) =d_s\}} \right]\, .
%&= \sum_{q=1}^m \sum_{\substack{N_j \geq 1 \\ N_1+\ldots+N_q=m}} c_{\mathbf{N}} \sum_{1 \leq n_1 < \ldots < n_q \leq n} \sum_{\mathbf{a} \in (\Z^d)^q} \Ebb_\mu \left[ \prod_{j=1}^q (hf(T^{n_j}(\cdot),a_j))^{N_j} \mathbf{1}_{\{S_{n_j}=a_j\}} \right]\\
\end{align*}
Gathering the terms for which the $k_s$ (with their multiplicities) are the same,
we obtain 
\begin{equation}
\Ebb_\mu \left[ (\Zcal_nh_n)^m \right] 
= \sum_{q=1}^m \sum_{\substack{N_j \geq 1 \\ N_1+\cdots+N_q=m}} c_{\mathbf{N}} A_{n;q,\mathbf{N}}\, , \label{momentAn}
\end{equation}
where $\mathbf{N}=(N_j)_{1 \leq j \leq q}$, where $c_{\mathbf{N}}$ is the cardinal of the set of maps 
$\phi: \ \{1,\ldots,m\} \to \{1,\ldots,q\}$ such that $| \phi^{-1} (\{j\}) | = N_j$ for all $j\in\{1,\ldots,q\}$, 
% For all $n \geq 1$, for all $1 \leq q \leq m$ and for all such that $N_j \geq 1$ and $\sum_{j=1}^q N_j = m$
and where
\begin{align*}
A_{n;q;\mathbf{N}} 
& := \sum_{1 \leq n_1< \ldots <n_q \leq n} \sum_{\mathbf{a} \in (\Z^2)^q} \Ebb_\mu \left[ \prod_{j=1}^q \left( h_n(T^{n_j}(\cdot),a_j)^{N_j} \mathbf{1}_{\{S_{n_j} F (\cdot)=a_j\}} \right) \right] \\
& = \sum_{1 \leq n_1< \ldots <n_q \leq n} \sum_{\mathbf{a} \in (\Z^2)^q} \Ebb_\mu \left[ \prod_{j=1}^q \left( h_n(T^{n_j}(\cdot),a_j)^{N_j} \mathbf{1}_{\{S_{n_j} F(\cdot) -S_{n_{j-1}} F(\cdot) = a_j-a_{j-1}\}} \right) \right]\\
& = \sum_{\mathbf{a} \in (\Z^2)^q} \left[ \sum_{\boldsymbol{\ell} \in E_{q,n}} \Ebb_\mu \left[ \prod_{j=1}^q \left( h_n(T^{\ell_j}(\cdot),a_j)^{N_j}\mathbf{1}_{\{S_{\ell_j} F (\cdot) = a_j-a_{j-1}\}}\right) \circ T^{\ell_1+\ldots+\ell_{j-1}} (\cdot) \right] \right]\, ,
\end{align*}
with the notations $\mathbf a=(a_1,\ldots,a_q)$, $n_0:=0$, $a_0:=0$ and
\begin{equation*}
E_{q,n}
= \left\{ \boldsymbol{\ell}=({\ell}_1,\ldots,\ell_q) \in \{1, \ldots ,n\}^q\ :\ \sum_{j=1}^q \ell_j \leq n \right\}\, ;
\end{equation*}
$\ell_j$ corresponds to $n_j-n_{j-1}$. As in the proof of Lemma~\ref{lem:0}, 
for all $\ell \in \N$ and $a \in \Z^2$, we define operators $Q_{\ell,a}$ and 
$\widetilde{Q}_{\ell,a,b,N,n}$ acting on $\Bcal$ by:
\begin{equation*}
\begin{array}{lll}
 Q_{\ell,a} (G) (x) & := & P^\ell \left(\mathbf{1}_{\{S_\ell F (x)=a\}} G\right) (x)\, , \\
 \widetilde Q_{\ell,a,b,N,n} (G) (x) & := & h_n(x,a)^{N}  Q_{\ell,a-b} (G) (x) \, .
\end{array}
\end{equation*}
Using $\Ebb_\mu[\cdot]=\Ebb_\mu [P^{\ell_1+\ldots+\ell_q}(\cdot)]$ and using repeatedly 
$P^k (G \circ T^k \cdot H) = G \cdot P^k (H)$, we obtain
\begin{equation}\label{eq:AExpressionQ}
A_{n;q;\mathbf{N}} 
= \sum_{\mathbf{a} \in(\Z^2)^q}  \sum_{\boldsymbol{\ell} \in E_{q,n}} \Ebb_\mu \left[ \widetilde Q_{\ell_q,a_q,a_{q-1},N_q,n} \cdots \widetilde Q_{\ell_1,a_1,0,N_1,n}(\mathbf{1}) \right] .
\end{equation}
We further split the operators $Q_{\ell,a}$:
\begin{equation}\label{devQ}
Q_{\ell,a} 
= Q_{\ell,a}^{(0)} + Q_{\ell,a}^{(1)} 
\text{ with }
Q_{\ell,a}^{(0)} 
:= \Phi(0)\frac{\Ebb_\mu[\cdot]}{\ell}\, .
\end{equation}

We assume without loss of generality that $\eta = \varkappa/4 \leq 1$.
Note that 
\begin{equation}\label{eq:Q1}
\norm{ Q^{(1)}_{\ell,a}}{\Lcal (\Bcal, \Bcal)} = O((1+|a|^
{2\eta}
%{\frac{\varkappa}{2}}
)\ell^{-1-\eta})
\end{equation}
by Hypothesis~\eqref{hyp:0} and using the fact that $|\Phi(x)-\Phi(0)|\le \mathfrak c \min(x^2,1)\le \mathfrak c x^{2\eta}$ for some $\mathfrak c>0$. 
Thus, for all $N \geq 1$,
\begin{equation*}
  \norm{h_n(\cdot,a)^N Q^{(1)}_{\ell,a-b}}{\Lcal (\Bcal, \Bcal)} 
  \leq C\frac{1+|a-b|^{\frac{\varkappa}{2}}}{\afrak_n^{\frac{N}{2}}\ell^{1+\eta}}
  \left(\frac{\mathbf{1}_0(a)}{\afrak_n^{\frac{N}{2}}} + \norm{f(\cdot,a)\times\cdot}{\Lcal(\Bcal,\Bcal)}^N\right) .
\end{equation*}
We introduce these operators $Q_{\ell,a}^{(0)}$ and $Q_{\ell,a}^{(1)}$ 
into~\eqref{eq:AExpressionQ}, creating new data we need to track: 
the index of the operator we use at each point in the weighted path. 
Fix $n$, $q$ and $\mathbf{N}$. Given $\boldsymbol{\varepsilon} = (\varepsilon_1, \ldots, \varepsilon_q) \in \{0,1\}^q$ 
and $s \in \Z^2$, write:
\begin{align*}
B_{s,\boldsymbol{\ell}, \mathbf{N}}^{\boldsymbol{\varepsilon}} (G)
& := \sum_{\substack{a_0,\ldots,a_q \in \Z^2 \\ a_0=s}} h_n(\cdot,a_q)^{N_q} Q_{\ell_q,a_q-a_{q-1}}^{(\varepsilon_q)} \cdots h_n(\cdot,a_1)^{N_1}\, Q_{\ell_1,a_1-a_0}^{(\varepsilon_1)} (G) \, , \\
b_{s, \boldsymbol{\ell}, \mathbf{N}}^{\boldsymbol{\varepsilon}} (G)
& := \sum_{\mathbf{a} \in (\Z^2)^q}  \Ebb_\mu \left[ h_n(\cdot,a_q)^{N_q}Q_{\ell_q,a_q-a_{q-1}}^{(\varepsilon_q)} \ldots h_n(\cdot,a_1)^{N_1}\, Q_{\ell_1,a_1-s}^{(\varepsilon_1)} (G) \right]
= \Ebb_\mu \left[B_{s,\boldsymbol{\ell},\mathbf{N}}^{\boldsymbol{\varepsilon}} (G)\right], \\
A_{n;q;\mathbf{N}}^{\boldsymbol{\varepsilon}} 
& := \sum_{\boldsymbol{\ell} \in E_{q,n}} b_{0,\boldsymbol{\ell}, \mathbf{N}}^{\boldsymbol{\varepsilon}}(\mathbf{1}) \, ,
\end{align*}
so that:
\begin{equation*}
A_{n;q;\mathbf{N}} 
= \sum_{\boldsymbol{\varepsilon}\in\{0,1\}^q} A_{n;q;\mathbf{N}}^{\boldsymbol{\varepsilon}}
= \sum_{\boldsymbol{\varepsilon}\in\{0,1\}^q} \sum_{\boldsymbol{\ell} \in E_{q,n}} b_{0, \boldsymbol{\ell}, \mathbf{N}}^{\boldsymbol{\varepsilon}}(\mathbf{1}).
\end{equation*}

\smallskip

The goal is now to understand which combinatorial data $(\mathbf{N}, \boldsymbol{\varepsilon})$ is negligible 
as $n \to +\infty$, and which represent the majority of the $m$-th moment. The following properties are directly implied by the definitions.

\begin{properties}\label{rmk:proprietesbn}

Consider a single linear form $b_{s, \boldsymbol{\ell}, \mathbf{N}}^{\boldsymbol{\varepsilon}}$. 
For all $1 \leq i \leq q$, the terms on the right side of $Q_{\ell_i,a_i-a_{i-1}}^{(\varepsilon_i)}$ 
depend only on $a_1,\ldots,a_{i-1}$, and the terms on its left side only depend on $a_i,\ldots,a_q$. 
Hence:
\begin{enumerate}[(I)]
\item Since $Q_{\ell,a}^{(0)}$ does not depend on $a$, the value of $b_{s, (\ell_0,\boldsymbol{\ell}),(N_0,\mathbf{N})}^{(0,\boldsymbol\varepsilon)}$ 
does not depend on $s$. Without loss of generality, we shall choose $s$ to be $0$ when $\varepsilon_1=0$.

\smallskip

\item 
$b_{s, (\ell), (N)}^{(0)} (\cdot) = \frac{\Phi(0)}{\ell} \sum_{a \in \Z^2} \Ebb_{\mu}[h_n(\cdot,a)^N] \Ebb_\mu[\cdot]$ for all $\ell$, $N \geq 1$, so:
\begin{align*}
 b_{s, (\ell), (1)}^{(0)} (\cdot) & = \frac{\Phi(0) \alpha}{\ell \afrak_n} \Ebb_\mu[\cdot] \, , \\
 b_{s, (\ell), (2)}^{(0)} (\cdot) & = \frac{\Phi(0) \beta^2}{\ell \afrak_n} \left(\sum_{a \in \Z^2} \Ebb_{\mu}[f(\cdot,a)^2] \Ebb_\mu[\cdot]+O\left(\frac{1}{\sqrt{\afrak_n}} \right)\right) \, .
\end{align*}

\smallskip

\item $b_{s, (\boldsymbol{\ell},\ell_0,\boldsymbol{\ell}'),(\mathbf{N},N_0,\mathbf{N}')}^{(\boldsymbol{\varepsilon},0,\boldsymbol{\varepsilon}')} 
= \Ebb_\mu[B_{0, (\ell_0,\boldsymbol{\ell}'),(N_0,\mathbf{N}')}^{(0,\boldsymbol\varepsilon')}(\mathbf{1})] \Ebb_\mu [ B_{s, \boldsymbol{\ell},\mathbf{N}}^{\boldsymbol\varepsilon}(\cdot)]$,
i.e.
\begin{equation*}
b_{s, (\boldsymbol{\ell},\ell_0,\boldsymbol{\ell}'),(\mathbf{N},N_0,\mathbf{N}')}^{(\boldsymbol{\varepsilon},0,\boldsymbol{\varepsilon}')} (\cdot)
= b_{0, (\ell_0,\boldsymbol{\ell}'),(N_0,\mathbf{N}')}^{(0,\boldsymbol\varepsilon')}(\mathbf{1}) b_{s, \boldsymbol{\ell},\mathbf{N}}^{\boldsymbol\varepsilon}(\cdot) \, .
\end{equation*}

\smallskip

\item In particular, $b_{s, (\boldsymbol{\ell},\ell_0),(\mathbf{N},1)}^{(\boldsymbol{\varepsilon},0)} = \frac{\Phi(0) \alpha}{\ell \afrak_n} b_{s, \boldsymbol{\ell},\mathbf{N}}^{\boldsymbol\varepsilon}(\cdot)$, and:
\begin{align*}
b_{s, (\boldsymbol{\ell},\ell_0,\ell_0',\boldsymbol{\ell}'),(\mathbf{N},1,N_0',\mathbf{N}')}^{(\boldsymbol{\varepsilon},0,0,\boldsymbol{\varepsilon}')} 
& = b_{0, (\ell_0',\boldsymbol{\ell}'),(N'_0,\mathbf{N}')}^{(0,\boldsymbol\varepsilon')} (\mathbf{1}) b_{0, (\ell_0),(1)}^{(0)}(\mathbf{1}) b_{s, \boldsymbol{\ell},\mathbf{N}}^{\boldsymbol\varepsilon}(\cdot) \\
& = \frac{\Phi(0) \alpha}{\ell_0 \afrak_n} b_{0, (\ell_0',\boldsymbol{\ell}'),(N'_0,\mathbf{N}')}^{(0,\boldsymbol\varepsilon')} (\mathbf{1}) b_{s, \boldsymbol{\ell},\mathbf{N}}^{\boldsymbol\varepsilon}(\cdot) \, .
\end{align*}

\smallskip

\item \begin{equation*}
b_{s, (\ell_1, \ldots,\ell_q),(N_1,N_2, \ldots,N_q)}^{(0,1,\ldots,1)} (\mathbf{1})
= \frac{\Phi(0)}{\ell_1} \sum_{a_1 \in \Z^2}  b_{a_1, (\ell_2, \ldots,\ell_q),(N_2, \ldots,N_q)}^{(1,\ldots,1)} (h_n(\cdot,a_1)^{N_1}).
\end{equation*}
\end{enumerate}
\end{properties}

In particular, we can estimate the coefficients corresponding to $\boldsymbol\varepsilon = (0,1)$ and $\mathbf{N} = (1,1)$, 
which will play an important role later on.

\begin{lemma}
\label{lem:Borne01}
 
 Under the hypotheses of Theorem~\ref{thm:gene} and with the notations of its proof,
 \begin{equation}
 \label{eq:Evaluation01}
  b_{s, (\ell,\ell'),(1,1)}^{(0,1)}(\mathbf{1})
  = \frac{\Phi(0) \beta^2}{\afrak_n \ell} \int_{\widetilde{A}} f\circ\widetilde{T}^{\ell'}\cdot f\dd\widetilde{\mu}+ O\left(\frac{1}{\ell\ell'^{1+\eta}\afrak_n^\frac{3}{2}}\right) \, .
 \end{equation}
\end{lemma}

\begin{proof}[Proof of Lemma~\ref{lem:Borne01}]

Applying Point~$(V)$ of Properties~\ref{rmk:proprietesbn},
\begin{equation*}
 b_{s, (\ell,\ell'),(1,1)}^{(0,1)}(\mathbf{1}) 
 = \frac{\Phi(0)}{\ell} \sum_{a,b\in\Z^2} \Ebb_\mu \left[h_n(\cdot,b) Q_{\ell',b-a}^{(1)} (h_n(\cdot,a))\right] \, .
\end{equation*}
Recall that $h_n=\frac{\alpha}{\afrak_n} g + \frac{\beta}{\sqrt{\afrak_n}} f$ with $g (x,a) =\mathbf{1}_0(a)$. 
By Equation~\eqref{eq:Q1}, $\norm{ Q^{(1)}_{\ell,a}}{\Lcal (\Bcal, \Bcal)} = O((1+|a|^{2\eta})\ell^{-1-\eta})$, 
therefore
\begin{align}
 \frac{\Phi(0)}{\ell} \sum_{a,b\in\Z^2} \Ebb_\mu \left[\frac{\alpha}{\mathfrak a_n} g(\cdot,b) Q_{\ell',b-a}^{(1)} (h_n(\cdot,a))\right]
 & = O\left(\frac {1}{\ell\, \ell'{}^{1+\eta}\, \mathfrak a_n} \sum_{a\in\Z^2}(1+|a|^{2\eta}) \norm{h_n(\cdot,a)}{\Bcal}\right) \nonumber\\
 & = O\left(\frac {1}{\ell\, \ell'{}^{1+\eta}\, \mathfrak a_n} \sum_{a\in\Z^2}(1+|a|^{2\eta})  \left(\frac{\mathbf{1}_0(a)}{\mathfrak a_n}+\frac{\norm{f(\cdot,a)}{\Bcal}}{\sqrt{\mathfrak a_n}}\right)\right) \nonumber\\
 & = O\left(\frac {1}{\ell\, \ell'{}^{1+\eta}\, \mathfrak a_n^{\frac 32}} \right)\, ,
\label{eq:AAA1}
\end{align}
since $\sum_{a\in\Z^2}(1+|a|^{2\eta})\norm{f(\cdot,a)}{\Bcal}<\infty$. In the same way,
\begin{align*}
 \frac{\Phi(0)}{\ell} 
 & \sum_{a,b\in\Z^2} \Ebb_\mu \left[\frac{\beta}{\sqrt{\mathfrak a_n}} f(\cdot,b) Q_{\ell',b-a}^{(1)} (h_n(\cdot,a))\right]\\
 & = \frac{\Phi(0) \beta^2}{\afrak_n \ell}\sum_{a,b\in\Z^2} \Ebb_\mu \left[ f(\cdot,b) Q_{\ell',b-a}^{(1)} (f(\cdot,a))\right]
 + \frac{\Phi(0)\alpha\beta}{\ell\mathfrak a_n^{\frac 32}} \sum_{b\in\Z^2} \Ebb_\mu \left[ f(\cdot,b) Q_{\ell',b}^{(1)} (\mathbf{1})\right]\\
 & = \frac{\Phi(0) \beta^2}{\afrak_n \ell}\sum_{a,b\in\Z^2}\Ebb_\mu \left[ f(\cdot,b) Q_{\ell',b-a} (f(\cdot,a))\right]
 + \frac{\Phi(0)\alpha\beta}{\ell\mathfrak a_n^{\frac 32}} \sum_{b\in\Z^2} \norm{ f(\cdot,b)\times\cdot}{\Lcal(\Bcal,\Bcal)} O\left(\frac{1+|b|^{2\eta}}{\ell'{}^{1+\eta}}\right) \\
 & =\frac{\Phi(0) \beta^2}{\afrak_n \ell} \int_{\widetilde{A}} f\circ\widetilde{T}^{\ell'}\cdot f\dd\widetilde{\mu}+O\left(\frac {1}{\ell\, \ell'{}^{1+\eta}\, \mathfrak a_n^{\frac 32}}\right)\, ,
\end{align*}
where we used the fact that $\sum_{a,b\in\Z^2}\Ebb_\mu\left[f(\cdot,b)Q^{(0)}_{\ell',b-a}(f(\cdot,a))\right]=\frac{\Phi(0)}{\ell'}\left(\int_{\widetilde A}f \dd\widetilde\mu\right)^2=0$.
The claim follows from this, combined with Equation~\eqref{eq:AAA1}.
\end{proof}

Given a sequence $\boldsymbol\varepsilon \in \{0, 1\}^q$, we can iterate 
Point~$(III)$ of Properties~\ref{rmk:proprietesbn} to cut $b_{s, \boldsymbol{\ell},\mathbf{N}}^{\boldsymbol\varepsilon}$ 
into smaller pieces, for which $0$ may only appear at the beginning of the associated 
sequences of indices, and then use Point~$(V)$ to transform the heading $\varepsilon_i = 0$. 
Let $m_1<m_2<\cdots<m_K$ be the indices $i\in\{1,\ldots,q\}$ such that $\varepsilon_i=0$. 
We use the conventions that $m_{K+1}:=q+1$ and $\varepsilon_{q+1} := 0$, that 
$b_{s,\boldsymbol{\ell}, \mathbf{N}}^{\boldsymbol{\varepsilon}} \equiv 1$ if $q=0$, 
and that an empty product is also equal to $1$. Then:
\begin{align*}
b_{s,\boldsymbol{\ell}, \mathbf{N}}^{\boldsymbol{\varepsilon}}(\mathbf{1})
& = b_{s,(\ell_1,\ldots,\ell_{m_1-1}), (N_1,\ldots,N_{m_1-1})}^{(1,\ldots,1)}(\mathbf{1}) \prod_{i=1}^K b_{0,(\ell_{m_i}, \ldots,\ell_{m_{i+1}-1}),(N_{m_i}, \ldots,N_{m_{i+1}-1})}^{(0,1, \ldots ,1)}(\mathbf{1}) \\
& = (\Phi(0))^K  b_{s,(\ell_1,\ldots,\ell_{m_1-1}), (N_1,\ldots,N_{m_1-1})}^{(1,\ldots,1)}(\mathbf{1}) \\
& \hspace{2em} \times \prod_{i=1}^K \frac{1}{\ell_{m_i}} \sum_{a \in \Z^d}  b_{a, (\ell_{m_i+1}, \ldots,\ell_{m_{i+1}-1}),(N_{m_i+1}, \ldots,N_{m_{i+1}-1})}^{(1, \ldots ,1)}(h_n(\cdot, a)^{N_{m_i}}) \, .
\end{align*}
We sum over $\boldsymbol{\ell} \in E_{q,n}$, and get:
\begin{align}
\left| A_{n,q, \mathbf{N}}^{\boldsymbol{\varepsilon}} \right|
& \leq \sum_{\boldsymbol{\ell} \in \{1, \ldots, n\}^q} \left| b_{0,\boldsymbol{\ell}, \mathbf{N}}^{\boldsymbol{\varepsilon}}(\mathbf{1}) \right| \nonumber\\
& \leq (\Phi(0))^K \left( \sum_{\substack{(\ell_1,\ldots,\ell_{m_1-1}) \\ \in \{1, \ldots, n\}^{m_1-1}}} \left| b_{0,(\ell_1,\ldots,\ell_{m_1-1}), (N_1,\ldots,N_{m_1-1})}^{(1,\ldots,1)}(\mathbf{1}) \right| \right) \label{majoA1}\\
& \hspace{2em} \times \prod_{i=1}^K \left( \sum_{\substack{(\ell_{m_i}, \ldots,\ell_{m_{i+1}-1}) \\ \in \{1, \ldots, n\}^{m_{i+1}-m_i}}} \frac{1}{\ell_{m_i}} \left| \sum_{a \in \Z^d}  b_{a, (\ell_{m_i+1}, \ldots,\ell_{m_{i+1}-1}),(N_{m_i+1}, \ldots,N_{m_{i+1}-1})}^{(1, \ldots ,1)}(h_n(\cdot, a)^{N_{m_i}}) \right|\right). \label{majoA2}
\end{align}

Now, let us bound the terms~\eqref{majoA1} and~\eqref{majoA2}; our goal is to find conditions on the combinatorial data 
ensuring that these terms are negligible. Starting with~\eqref{majoA1}, 
\begin{align*}
& \norm{b_{0,(\ell_1,\ldots,\ell_{m_1-1}), (N_1,\ldots,N_{m_1-1})}^{(1,\ldots,1)}}{\Bcal^*} \\
& \hspace{4em} = O\left(\frac{(\ell_1\ldots\ell_{m_1-1})^{-1-\eta}}{\afrak_n^{\frac {N_1+\ldots+N_{m_1-1}}2}} \sum_{a_1,\ldots,a_{m_1-1}\in\Z^2}\prod_{j=1}^{m_1-1}(1+|a_j-a_{j-1}|^{2\eta}
%{\frac{\varkappa}{2}}
)\left(\mathbf{1}_0(a_j)+\norm{ f(\cdot,a_j)\times\cdot}{\Lcal(\Bcal,\Bcal)}^{N_j}\right)\right) \\
& \hspace{4em} = O\left(\frac{(\ell_1\ldots\ell_{m_1-1})^{-1-\eta}}{\afrak_n^{\frac {N_1+\ldots+N_{m_1-1}}2}} \prod_{j=1}^{m_1-1} \sum_{a\in\Z^2}(1+|a|^{4\eta}%\varkappa
)\left(\mathbf{1}_0(a)+\norm{f(\cdot,a)\times\cdot}{\Lcal(\Bcal,\Bcal)}^{N_j}\right)\right) \, .
\end{align*}
Therefore, \eqref{majoA1} is bounded, and converges to $0$ as $n \to + \infty$ if $m_1 \neq 1$. 
Focusing now on~\eqref{majoA2}, 
\begin{align*}
& \norm{\sum_{a\in \Z^d}  b_{a,(\ell_{m_i+1},\ldots,\ell_{m_{i+1}-1}), (N_{m_i+1},\ldots,N_{m_{i+1}-1})}^{(1,\ldots,1)}(h_n(\cdot,a)^{N_{m_i}})}{\Bcal^*} \\
& \hspace{4em} = O \left(\frac{(\ell_{m_i+1}\ldots\ell_{m_{i+1}-1})^{-1-\eta}}{\afrak_n^{\frac {N_{m_i}+\ldots+N_{m_{i+1}-1}}2}} \sum_{a_0,\ldots,a_{m_{i+1}-m_i-1}\in\Z^2}\prod_{j=1}^{m_{i+1}-m_i-1}(1+|a_j-a_{j-1}|^{2\eta
%\frac\varkappa 2
})\right.\\
& \hspace{6em} \left. \times \prod_{k=0}^{m_{i+1}-m_i-1}\left(\mathbf{1}_0({a_k})+\norm{f(\cdot,a_k)\times\cdot}{\Lcal(\Bcal,\Bcal)}^{N_{m_i+k}}\right)\right)\\
& \hspace{4em} = O\left(\frac{(\ell_{m_i+1}\ldots\ell_{m_{i+1}-1})^{-1-\eta}}{\afrak_n^{\frac {N_{m_i}+\ldots+N_{m_{i+1}-1}}2}} \prod_{j=m_i}^{m_{i+1}-1}\sum_{a\in\Z^2}(1+|a|^{4\eta
%\varkappa
})\left(\mathbf{1}_0(a)+\norm{f(\cdot,a)\times\cdot}{\Lcal(\Bcal,\Bcal)}^{N_j}\right)\right)\, .
\end{align*}
Therefore the $i$-th term appearing in~\eqref{majoA2} is in 
\begin{equation}
 \label{eq:ControleNmi}
 O\left( \left( \sum_{\ell_{m_i} \in \{1, \ldots, n\}} \frac{1}{\ell_{m_i}} \right) \afrak_n^{-\frac {N_{m_i}+\ldots+N_{m_{i+1}-1}}2}\right)
 = O\left( \afrak_n^{1-\frac {N_{m_i}+\ldots+N_{m_{i+1}-1}}{2}}\right) \, .
\end{equation}
If $m_{i+1} = m_i+1$, then the $i$-th term in~\eqref{majoA2} is bounded by Point~$(II)$ of Properties~\ref{rmk:proprietesbn}. 
If $m_{i+1} \geq m_i+2$, then $N_{m_i}+\ldots+N_{m_{i+1}-1} \geq 2$, so the $i$-th term in~\eqref{majoA2} is still bounded 
by Equation~\eqref{eq:ControleNmi}. 
Furthermore, if $N_{m_i}+\ldots+N_{m_{i+1}-1} \geq 3$ for some $i$, 
then the $i$-th term converges to $0$, and thus $A_{n,q, \mathbf{N}}^{\boldsymbol{\varepsilon}}$ 
converges to $0$ as $n \to + \infty$ by Equation~\eqref{eq:ControleNmi}. Hence, $A_{n,q, \mathbf{N}}^{\boldsymbol{\varepsilon}}$ may not 
converge to $0$ only if $m_1 = 1$ and $N_{m_i}+\ldots+N_{m_{i+1}-1} \leq 2$ for all $i$. To 
sum up, if:
\begin{equation}\label{dominant}
m_1=1 \mbox{ and for all } 1 \leq i \leq K, \text{ either }  \left\{ 
\begin{array}{rclcrcl}
      m_{i+1} & = & m_i+1 & \text{ and } & N_{m_i} & = & 1 \\
      m_{i+1} & = & m_i+1 & \text{ and } & N_{m_i} & = & 2 \\
      m_{i+1} & = & m_i+2 & \text{ and } & N_{m_i} & = & N_{m_i+1}=1 
\end{array}\right. ,
\end{equation}
then $(A_{n,q, \mathbf{N}}^{\boldsymbol{\varepsilon}})_{n \geq 0}$ is bounded; otherwise, 
$(A_{n,q, \mathbf{N}}^{\boldsymbol{\varepsilon}})_{n \geq 0}$ converges to $0$. In particular, 
$\Ebb_\mu \left[\Zcal_n (h_n)^m \right]$ is bounded, and we only need to take into account 
the data $(\mathbf{N},\boldsymbol{\varepsilon})$ satisfying Condition~\eqref{dominant}, which can be rewritten:
\begin{itemize}
\item $N_i \in\{1, 2\}$; 
\item $N_i=2\ \Rightarrow\ \varepsilon_i=0$; 
\item $\varepsilon_i=1\ \Rightarrow\ i\ge 2,\ N_i=N_{i-1}=1,\  \varepsilon_{i-1}=0$.
\end{itemize}
We shall call such couples $(\mathbf{N},\boldsymbol{\varepsilon})$ \textit{admissible}.
Given $1 \leq q \leq m$, let $\Gcal(q)$ be the set of admissible 
$(\mathbf{N},\boldsymbol{\varepsilon})=((N_1,\ldots,N_q),(\varepsilon_1,\ldots,\varepsilon_q)) \in\{1,2\}^q\times\{0,1\}^q$. 
For $(\mathbf{N},\boldsymbol{\varepsilon})\in\Gcal(q)$, we set:
\begin{itemize}
 \item $\Jcal_2 := \{i\in\{1,...,q\}\ :\ \varepsilon_i=0, N_i=2\}$;
 \item $\Jcal_1 := \{i\in\{1,...,q\}\ :\ (\varepsilon_i, \varepsilon_{i+1})= (0,0), N_i=1 \}$;
 \item $\Jcal_{1,1} := \{i\in\{1,...,q-1\}\ :\ (\varepsilon_i, \varepsilon_{i+1})= (0,1), (N_i, N_{i+1}) = (1,1) \}$\, ,
\end{itemize}
recalling the convention $\varepsilon_{q+1}=0$. 

\smallskip

For instance, the data $\mathbf{N} = (1,1,1,2,1,1,2,2,1,1)$, $\boldsymbol{\varepsilon} = (0,0,1,0,0,1,0,0,0,0)$ 
is admissible, as it can be decomposed in blocs as follows:
\begin{equation*}
 \begin{array}{|c|c|cc|c|cc|c|c|c|c|}
  \hline
  \mathbf{N} & 1 & 1 & 1 & 2 & 1 & 1 & 2 & 2 & 1 & 1 \\
  \hline
  \boldsymbol{\varepsilon} & 0 & 0 & 1 & 0 & 0 & 1 & 0 & 0 & 0 & 0 \\
  \hline
 \end{array}
 .
\end{equation*}
For this example, $\Jcal_2 = \{4, 7, 8\}$, $\Jcal_1 = \{1, 9, 10\}$ and $\Jcal_{1,1} = \{2, 5\}$.

\smallskip

Then:
\begin{equation*}
 b_{0;\boldsymbol{\ell},\mathbf{N}}^{\boldsymbol{\varepsilon} }(\mathbf{1})
 = \left(\prod_{i \in \Jcal_2} b_{(\ell_i),(2)}^{(0)} (\mathbf{1}) \right) 
   \left(\prod_{i\in\Jcal_1} b_{(\ell_i),(1)}^{(0)} (\mathbf{1}) \right)
   \left(\prod_{i\in\Jcal_{1,1}} b_{(\ell_i,\ell_{i+1}),(1,1)}^{(0,1)}(\mathbf{1})\right).
\end{equation*}
Note that $m = 2|\Jcal_2|+2|\Jcal_{1,1}| + |\Jcal_1|$ while $q = |\Jcal_2|+2|\Jcal_{1,1}| + |\Jcal_1|$; 
in particular, $|\Jcal_2| = m-q$. 
Due to Point~$(II)$ in Properties~\ref{rmk:proprietesbn} and Lemma~\ref{lem:Borne01}, we obtain:
\begin{align}\label{FORMULEA}
A_{n;q;\mathbf{N}}^{\boldsymbol\varepsilon}
& = \sum_{\boldsymbol{\ell}\in E_{q,n}} \left(\prod_{i \in \Jcal_2} \frac{\Phi(0)\beta^2\sum_{a\in\Z^2} \Ebb_{\mu}[f(\cdot,a)^2]}{ {\ell_i}\afrak_n} \right)\, \left( \prod_{i\in\Jcal_1} \frac{\Phi(0) \alpha}{\ell_i \afrak_n} \right) \nonumber \\
& \hspace{2em} \times \left( \prod_{i\in\Jcal_{1,1}} \frac{\Phi(0)\beta^2}{\ell_i \afrak_n} \int_{\widetilde{A}} f\circ\widetilde{T}^{\ell_i} \cdot f\dd\widetilde{\mu} \right)+o(1) \nonumber \\
& = \left(\frac{\Phi(0)}{\afrak_n}\right)^{\frac{m+|\Jcal_1|}{2}} \beta^{m-|\Jcal_1|} \alpha^{|\Jcal_1|} \sum_{\boldsymbol{\ell}\in  E_{q,n}} \left( \prod_{i \in \Jcal_2} \frac{\sum_{a\in\Z^2} \Ebb_{\mu}[f(\cdot,a)^2]}{{\ell_i}} \right) \nonumber \\
& \hspace{2em} \times \left( \prod_{i\in\Jcal_1} \frac{1}{\ell_i} \right) \left( \prod_{i\in\Jcal_{1,1}} \frac{1}{\ell_i} \int_{\widetilde{A}} f\circ\widetilde{T}^{\ell_i} \cdot f\dd\widetilde{\mu} \right) + o(1) \nonumber \\
& = \left(\frac{\Phi(0)}{\afrak_n}\right)^{\frac{m+|\Jcal_1|}{2}} \beta^{m-|\Jcal_1|} \alpha^{|\Jcal_1|} \left(\int_{\widetilde{A}} f^2 \dd \widetilde{\mu} \right)^{|\Jcal_2|} \nonumber \\
& \hspace{2em} \times \sum_{\ell_1,\ldots,\ell_{|\Jcal_{1,1}|}\ge 1}\left[\left( \prod_{i\in\Jcal_{1,1}} \int_{\widetilde{A}} f\circ\widetilde{T}^{\ell_i} \cdot f\dd\widetilde{\mu} \right) \sum_{\boldsymbol{\ell'}\in E_{q-|\Jcal_{1,1}|,n-\sum_{i=1}^{|\Jcal_{1,1}|}\ell_i}}\prod_{i=1}^{q-|\Jcal_{1,1}|} \frac{1}{\ell'_i} \right] + o(1)
\end{align}
Due to~\cite[Lemma 3.7]{PeneThomine:2019}, for all $\ell_1,\ldots,\ell_{|\Jcal_{1,1}|} \geq 1$, as $n \to +\infty$,
\begin{equation*}
 \sum_{\boldsymbol{\ell'}\in E_{q-|\Jcal_{1,1}|,n-\sum_{i=1}^{|\Jcal_{1,1}|}\ell_i}}\prod_{i=1}^{q-|\Jcal_{1,1}|}{\ell'_i}^{-1} 
 \sim \afrak_n^{q-|\Jcal_{1,1}|} 
 = \afrak_n^{\frac{m+|\Jcal_1|}{2}} \, .
\end{equation*}
%
%Moreover since $m=2(m-q)+2|\Jcal_{1,1}|+|\Jcal_1|$, we also have
%$$m-q+|\Jcal_{1,1}|+|\Jcal_1|=q-|\Jcal_{1,1}|\quad (i.e.\ 2q=m+2|\Jcal_{1,1}|+|\Jcal_1|).$$
Hence, by the dominated convergence theorem, 
\begin{equation*}
A_{n;q;\mathbf{N}}^{\boldsymbol\varepsilon} 
= \Phi(0)^{\frac{m+|\Jcal_1|}{2}} \beta^{m-|\Jcal_1|} \alpha^{|\Jcal_1|} \left(\int_{\widetilde{A}} f^2 \dd \widetilde{\mu} \right)^{|\Jcal_2|}
\left( \sum_{\ell \geq 1}\int_{\widetilde{A}} f\circ\widetilde{T}^\ell \cdot f\dd\widetilde{\mu}\right)^{|\Jcal_{1,1}|} + o(1)\, .
\end{equation*}
If $(\mathbf{N},\boldsymbol{\varepsilon})$ is admissible, then 
$c_{\mathbf{N}}=2^{-|\Jcal_2|} m!$. Applying Equation~\eqref{momentAn}, we obtain
\begin{equation*}
\Ebb_\mu \left[ \Zcal_n (h_n)^m \right]
= \sum_{q=1}^m \sum_{(\mathbf{N},\boldsymbol{\varepsilon}) \in \Gcal(q)} c_{\mathbf{N}} A_{n;q;\mathbf{N}}^{\boldsymbol\varepsilon} + o(1)
= m! \sum_{q=1}^m 2^{-|\Jcal_2|} \sum_{(\mathbf{N},\boldsymbol{\varepsilon}) \in \Gcal(q)} A_{n;q;\mathbf{N}}^{\boldsymbol\varepsilon} + o(1) \, .
\end{equation*}
Let $r := 2|\Jcal_{1,1}|+2|\Jcal_2|$ and $s := |\Jcal_2|$. Note that $r$ is even, $s \leq r/2$ and $r \leq m$. 
We split the later sum depending on the value of $r$, and then depending on the value of $s = m-q$. Note that, 
once $r$ and $s$ are fixed, the number of admissible $(\mathbf{N},\boldsymbol{\varepsilon})$ 
such that $r = 2|\Jcal_{1,1}|+2|\Jcal_2|$ and $s = |\Jcal_2|$ is $\binom{m-r/2}{r/2} \cdot \binom{r/2}{s}$. We get:
\begin{align*}
\lim_{n \to + \infty} & \Ebb_\mu \left[ \Zcal_n (h_n)^m \right] \\
& = m! \sum_{\substack{0 \leq r \leq m \\ r \in 2\Z}} \sum_{s=0}^{r/2} 2^{-s} \sum_{(\mathbf{N},\boldsymbol{\varepsilon}) \in \Gcal(m-s)} \Phi(0)^{m-r/2} \beta^r \alpha^{m-r} \left(\int_{\widetilde{A}} f^2 \dd \widetilde{\mu} \right)^s
\left( \sum_{\ell \geq 1}\int_{\widetilde{A}} f\circ\widetilde{T}^\ell \cdot f\dd\widetilde{\mu}\right)^{r/2-s} \\
& = m! \sum_{\substack{0 \leq r \leq m \\ r \in 2\Z}} \binom{m-r/2}{r/2} \Phi(0)^{m-r/2} \beta^r \alpha^{m-r} \sum_{s=0}^{r/2} \binom{r/2}{s} 2^{-s} \left(\int_{\widetilde{A}} f^2 \dd \widetilde{\mu} \right)^s
\left( \sum_{\ell \geq 1}\int_{\widetilde{A}} f\circ\widetilde{T}^\ell \cdot f\dd\widetilde{\mu}\right)^{r/2-s}  \\
& = m! \sum_{\substack{0 \leq r \leq m \\ r \in 2\Z}} \binom{m-r/2}{r/2} \Phi(0)^{m-r/2} \beta^r \alpha^{m-r} \left( \frac{\widetilde{\sigma}^2 (f)}{2} \right)^{r/2} \\
& = \sum_{\substack{0 \leq r \leq m \\ r \in 2\Z}} \binom{m}{r} \alpha^{m-r} \left[(m-r/2)! \Phi (0)^{m-r/2} \right] \frac{r!}{2^{r/2} (r/2)!} \left( \beta \widetilde{\sigma} (f) \right)^r \\
& = \sum_{r=0}^m \binom{m}{r} \alpha^{m-r} \Ebb \left[ (\Phi (0)\Ecal)^{m-r/2} \right] \Ebb \left[ \left( \beta \widetilde{\sigma} (f) \Ncal \right)^r \right] \\
& = \Ebb \left[ \left( \alpha \Phi (0)\Ecal + \beta \sqrt{\Phi(0)
\Ecal
} \widetilde{\sigma} (f) \Ncal \right)^m \right]\, ,
\end{align*}
where $\Ecal$ has a standard exponential distribution, $\Ncal$ a standard Gaussian distribution, and $\Ecal$, $\Ncal$ are independent. 
This finishes the proof of Theorem~\ref{thm:gene} for $t=1$.
\end{proof}

\subsection{Functional convergence}

We finish the proof of Theorem~\ref{thm:gene}, by extending the distributional limit theorem (for $t=1$) to 
a functional limit theorem. This is made easier by the fact that, in dimension $2$, the local time at step $n$ 
is of the order of $\ln (n)$, which has slow variation.

\begin{proof}[End of the proof of Theorem~\ref{thm:gene}]

A crucial observation is given by the next lemma:

\begin{lemma}
\label{controleMoments0}

Under Hypothesis~\ref{hyp:0}, there exists $C>0$ such that
for every $f : \widetilde{A} \to \R$, for every $0<T_1<T_2$ and every $n \geq T_1^{-1}$,
\begin{equation}
\norm{\sup_{t\in (T_1,T_2)}\left| \widetilde S_{nt}f-\widetilde S_{nT_1}f\right|}{\Lbb^1(\widetilde T_*(\mu\otimes\delta_0))}
\le C\sum_{a \in \Z^2} \norm{f (\cdot, a)}{\Lbb^{p^*} (A, \mu)}\log\frac{\lceil nT_2\rceil }{\lfloor nT_1\rfloor} \,.
\end{equation}
\end{lemma}

\begin{proof}

Assume first that $p=1$. Let $c_a:=\norm{f(\cdot,a)}{\infty}$ and set $h_0(x,a):=c_a$. Using Hypothesis~\ref{hyp:0}, 
there exists a constant $C>0$ such that:
\begin{align*}
 \norm{\widetilde S_k h_0-\widetilde S_j h_0}{\Lbb^1 (\widetilde T_*(\mu\otimes\delta_0))}
 & \leq \sum_{a\in\Z^2}c_a \sum_{m=j+1}^k \mu(S_m F=a) \\
 & \leq \sum_{a\in\Z^2}c_a \sum_{m=j+1}^k \Ebb_\mu\left[Q_{m,a} (\mathbf{1}) \right]
 \leq C \sum_{a\in\Z^2}c_a \sum_{m=j+1}^k \frac{1}{m} \, .
\end{align*}
Since $|f|\leq h_0$, for every $n \geq T_1^{-1}$.
\begin{align*}
 \norm{\sup_{t\in (T_1,T_2)}\left| \widetilde S_{nt}f-\widetilde S_{nT_1}f\right|}{\Lbb^1(\widetilde T_*(\mu\otimes\delta_0))} 
 & \leq \norm{ \widetilde S_{\lceil nT_2\rceil}h_0-\widetilde S_{\lfloor nT_1\rfloor}h_0 }{\Lbb^1(\mu\otimes\delta_0)} \\
 & \leq C\sum_{a\in\Z^2}c_a\log\frac{\lceil nT_2\rceil }{\lfloor nT_1\rfloor}\, .
\end{align*}

When $p>1$, using again Hypothesis~\ref{hyp:0}, we get:
\begin{align*}
\norm{\sup_{t\in (T_1,T_2)}\left|\widetilde{S}_{nt} f-\widetilde{S}_{nT_1} f\right|}{\Lbb^1(\widetilde T_*(\mu\otimes\delta_0))} 
& \leq \sum_{a\in\Z^2} \sum_{\ell=\lfloor nT_1\rfloor}^{\lceil nT_2\rceil} \norm{f (\cdot, a) P^\ell (\mathbf{1}_{S_\ell F = a}) }{\Lbb^1 (A, \mu)} \\
& \leq C' \sum_{a\in\Z^2} \norm{f (\cdot, a)}{\Lbb^{p^*} (A, \mu)} \sum_{\ell=\lfloor nT_1\rfloor}^{\lceil nT_2\rceil} \norm{Q_{\ell,a} (\mathbf{1})}{\Bcal} \\
& \leq C \left( \sum_{a\in\Z^2} \norm{f (\cdot, a)}{\Lbb^{p^*} (A, \mu)} \right) \sum_{\ell=\lfloor nT_1\rfloor}^{\lceil nT_2\rceil} \frac{1}{\ell} \qedhere
\end{align*}
\end{proof}
%Let $f$ be as in the hypotheses of Theorem~\ref{thm:gene}. and take $g_0 (x,a) = \mathbf{1}_0 (a)$. Let $0 < T_1 < T_2$. By Lemma~\ref{controleMoments0},
%
%\begin{align*}
% \norm{\sup_{t\in [T_1,T_2]}\left| \widetilde{S}_{nt}f- \widetilde{S}_{nT_1}f\right|}{\Lbb^1(\widetilde{A}, \mu\otimes\delta_0)} 
% & \leq C \left( \sum_{a \in \Z^2} \norm{f (\cdot, a)}{\Lbb^{p^*} (A, \mu)} \right) \ln \left( \frac{\lceil nT_2\rceil}{\lfloor nT_1\rfloor} \right) \\
% & \leq C' \left( \sum_{a \in \Z^2} \norm{f (\cdot, a)}{\Lbb^{p^*} (A, \mu)} \right) \ln \left( \frac{T_2}{T_1} \right),
%\end{align*}
%
%for every $n \geq T_1^{-1}$. Of course, the same holds true if we replace $f$ by $g_0$, so that:
%
Lemma~\ref{controleMoments0} implies that
\begin{equation*}
 \left( \sup_{t\in [T_1,T_2]}\left| \frac{\widetilde{S}_{nt}g_0-\widetilde{S}_{nT_1}g_0}{\ln (n)} \right|, \, 
 \sup_{t\in [T_1,T_2]}\left| \frac{\widetilde{S}_{nt}f-\widetilde{S}_{nT_1}f}{\sqrt{\ln (n)}} \right| \right)
\end{equation*}
converges in probability to $(0,0)$ with respect to $\mu\otimes\delta_0$. Hence, as $n$ goes to $+\infty$,
\begin{equation}\label{cv0}
\left(\frac{\widetilde{S}_{nt} g_0}{\ln (n)}, \, \frac{\widetilde{S}_{nt} f}{\sqrt{\ln (n)}}\right)_{t\in[T_1,T_2]}
\longrightarrow \left(\Phi(0)\Ecal, \, \widetilde{\sigma} (f) \sqrt{\Phi(0)\Ecal}\, \Ncal\right)_{t\in[T_1,T_2]}, 
\end{equation}
where the convergence is in distribution in $\Ccal([T_1,T_2], \R)$ with respect to $\mu\otimes\delta_0$.

\smallskip

Hence, the conclusion of Theorem~\ref{thm:gene} holds for $f$ and $g_0$, and where the convergence in distribution 
is with respect to $\mu\otimes\delta_0$. By~\cite[Theorem~1]{Zweimuller:2007}, the convergence in distribution 
actually holds with respect to any absolutely continuous probability measure. Finally, let us take 
any $g \in \Lbb^1 (\widetilde{A}, \widetilde{\mu})$. Since the system $(\widetilde{A},\widetilde{\mu},\widetilde{T})$ 
is conservative and ergodic, Hopf's ergodic theorem ensures that, $\widetilde{\mu}$-almost everywhere, 
$\widetilde{S}_t g \sim \int_{\widetilde{A}}g\dd\widetilde{\mu} \cdot (\widetilde{S}_t g_0)$, 
so the convergence in distribution of Equation~\eqref{cv0} also holds for $g$. 
\end{proof}

%\begin{equation*}
% \sup_{t\in[T_1,T_2]}\left|\frac {\widetilde{S}_{nt}g}{\ln (n)}-\frac{\int_{\widetilde{A}}g\dd\widetilde{\mu}}{\int_{\widetilde{A}}h_0\dd\widetilde{\mu}}\frac{\widetilde{S}_{nt}h_0}{\ln (n)}\right|=\sup_{u\ge nT_1}
%|\epsilon_{u}| \times  \frac{\widetilde{S}_{nT_2}h_0}{\ln (n)}\, ,
%\end{equation*}
%%
%which converges in probability to 0 as $n$ goes to infinity (using the fact that 
%$(\frac{\widetilde{S}_{nT_2}h_0}{\ln (n)})_n$ converges in distribution). 
%Combined with Equation~\eqref{cv0}, this implies the convergence in 
%distribution~\eqref{eq:ConvergenceDistributionGene} in $\Ccal([T_1,T_2], \R)$ with respect to 
%$\mu\times \delta_0$ and so with respect to any probability measure absolutely 
%continuous with respect to $\widetilde{\mu}$, due to .

\section{Limit theorem for flows}\label{sec:flow}

We now focus on the results for suspension flows over maps with good spectral properties.

\subsection{General theorem for suspension semiflows}

We begin by deducing Theorem~\ref{thm:geneflot} from Theorem~\ref{thm:gene}.

\begin{proof}[Proof of Theorem~\ref{thm:geneflot}]

Let $\phi$ be as in the hypotheses of Theorem~\ref{thm:geneflot}. 
Take $\psi (x, a, u) := \tau(x)^{-1} \mathbf{1}_0 (a)$ 
and $\mu_0 := \tau^{-1} (x) \dd \mu (x) \otimes \delta_0 (a) \otimes \dd u \in \Pcal (\widetilde{M})$. 
Let $0<s_1<s_2$. %Let $\Pbb_0$ be a probability measure absolutely continuous with 
%respect to $\widetilde{\nu}$, with density $h = \de \Pbb_0 / \de \widetilde{\nu}$.

\medskip
\textbf{From the transformation to the flow}
\smallskip

Let $\theta : \widetilde{\Mcal} \to \R$. Recall that we defined $G(\theta) (x,a)=\int_0^{\tau(x)} 
\theta
(x,a,t)\dd t$. 
Assume that:
\begin{equation*}
\sum_{a \in \Z^2} \norm{G(|\theta|) (\cdot, a)}{\Lbb^{p^*} (A, \mu)} 
< + \infty,
\end{equation*}
a condition satisfied by both the functions $\phi$ and $\psi$.

\smallskip
%\textcolor{red}{J'ai retabli la version precedente dans ce qui suit (avec la generalisation a $\Lbb^{p*}$)}

Recall that we set $n_t(x)=\max\{n\geq 0\, :\, S_n\tau(x)\leq t\}$. 
let $N_t(x):=n_t(x)+\frac{t-\sum_{k=0}^{n_t(x)-1}\tau\circ T^k(x)}{\tau(T^{n_t(x)}(x))}$, 
so that $\widetilde{S}_{N_t} \tau = t$. Then, for all $(x,a,u)\in\widetilde{\Mcal}$,
\begin{equation}
 \label{eq:BorneThetaMoyennise}
 \left| \widetilde{S}_t \theta (x,a,u) - \widetilde{S}_{N_t(x)} G(\theta) (x,a)\right|
 \leq G(|\theta|) (x,a) + \widetilde{S}_{n_{t+u} (x)-n_t(x)+1} G_u  (|\theta|) (\widetilde{T}^{n_t(x)}(x,a))\, .
\end{equation}

It is straightforward that $G(|\theta|) (x,a)/\sqrt{\ln (t)} \to 0$ as $t\to +\infty$. 
We need to control the last term in Equation~\eqref{eq:BorneThetaMoyennise}.

\smallskip

Since $(\widetilde{A},\widetilde{\mu},\widetilde{T})$ is ergodic, so is $(A,\mu,T)$, 
and thus, by Birkhoff's ergodic theorem, $\lim_{n\to +\infty}n^{-1} S_n \tau = \int_A\tau\dd\mu$ 
almost surely for $\mu$. Since $S_{n_t} \tau \leq t< S_{n_t+1} \tau$, 
we conclude that, $\mu$-almost surely, $n_t \sim \frac{t}{\int_A\tau\dd\mu}$ as $t$ goes to $+\infty$. 
Therefore, for $\widetilde{\nu}$-almost every $(x,a,u)\in\widetilde{\Mcal}$, there exists $t_0=t_0(x,u) \geq 0$ 
such that 
\begin{equation*}
 \frac{t}{2\int_A\tau\dd\mu} 
 \leq n_t (x) 
 \leq n_{t+u} (x)+1
 \leq \frac{2t}{\int_A\tau\dd\mu}
\end{equation*}
for every $t \geq t_0$. Then, on $\{t_0 \leq s_1 t \}$,
\begin{equation*}
 \sup_{s\in [s_1,s_2]} \widetilde{S}_{n_{ts+u} (x)-n_{ts} (x)+1} G(|\theta|) (\widetilde{T}^{n_{ts}(x)}(x,a))
 \leq \sup_{\frac {t s_1}{2\int_A\tau\dd\mu} \leq n \leq n+m \leq \frac{2 t s_2}{\int_A\tau\dd\mu}} \widetilde{S}_m G(|\theta|) (\widetilde{T}^n (x,a)). 
\end{equation*}
By Lemma~\ref{controleMoments0},
\begin{equation*}
 \norm{\mathbf{1}_{\{t_0 \leq s_1 t \}} \sup_{s\in [s_1,s_2]} \widetilde{S}_{n_{ts+u} (x)-n_{ts} (x)+1} G(|\theta|) (\widetilde{T}^{n_{ts}(x)}(x,a))}{\Lbb^1 (\widetilde{M}, \mu_0)} 
 \leq C \left( \sum_{a \in \Z^2} \norm{G(|\theta|) (\cdot, a)}{\Lbb^{p^*} (A, \mu)} \right) \ln \left( \frac{4 s_2}{s_1} \right).
\end{equation*}
Hence, the random variable 
\begin{equation*}
 \mathbf{1}_{\{t_0 \leq s_1 t \}} \frac{\sup_{s\in [s_1,s_2]}G
(|\theta|) (\widetilde{T}^{n_{ts+u}(x)}(x,a))}{\sqrt{\ln(t)}}
\end{equation*}
converges to $0$ in probability on $(\widetilde{M}, \mu_0)$, while the random variable 
\begin{equation*}
 \mathbf{1}_{\{t_0 > s_1 t \}} \frac{\sup_{s\in [s_1,s_2]} G
(|\theta|) (\widetilde{T}^{n_{ts+u}(x)}(x,a))}{\sqrt{\ln(t)}}
\end{equation*}
converges to $0$ almost surely on $(\widetilde{M}, \mu_0)$.

\smallskip

Applying the above discussion to the functions $\psi$ and $\phi$ respectively, the convergence in distribution in $\Ccal([s_1,s_2], \R)$, with respect to $\mu_0$, of 
\begin{equation*}
 \left(\frac{\widetilde{S}_{ts} \psi}{\ln (t)}, \,\frac{\widetilde{S}_{ts} \phi}{\sqrt{\ln (t)}}\right)_{s\in[s_1,s_2]}
\end{equation*}
is equivalent to the convergence in distribution in $\Ccal([s_1,s_2], \R)$, with respect to $\mu_0$, of 
\begin{equation*}
 (x,a,v) 
 \mapsto \left(\frac{\widetilde{S}_{N_{ts} (x)}G(\psi) (x,a)}{\ln (t)},\frac{\widetilde{S}_{N_{ts} (x)} G(\phi) (x,a)}{\sqrt{\ln (t)}}\right)_{s\in[s_1,s_2]}.
\end{equation*}
Since this last process depends only on $x$ (recall that $a = 0$ almost surely under $\mu_0$), 
this is equivalent to the convergence in distribution of the process 
\begin{equation*}
 x
 \mapsto \left(\frac{\widetilde{S}_{N_{ts} (x)} G(\psi) (x,0)}{\ln (t)},\frac{\widetilde{S}_{N_{ts} (x)}G(\phi)(x,0)}{\sqrt{\ln (t)}}\right)_{s\in[s_1,s_2]}
\end{equation*}
with respect to $\mu \otimes \delta_0$.

\medskip
\textbf{A time change}
\smallskip

It remains to prove the convergence in distribution of 
$\left(\frac{\widetilde{S}_{N_{ts}(\cdot)} G(\psi)(\cdot,0)}{\ln (t)}, \,\frac{\widetilde{S}_{N_{ts}(\cdot)} G(\phi)(\cdot,0)}{\sqrt{\ln (t)}}\right)_{s\in[s_1,s_2]}$ 
in $\Ccal([s_1,s_2], \R)$ with respect to $\mu$. The main idea is that 
this process is a time change (by $N_t$) of a discrete-time process, 
for which we can apply Theorem~\ref{thm:gene}.

\smallskip

We set $T_1:=\frac{s_1}{2 \int_A\tau\dd\mu}$ and $T_2:=\frac{2 s_2}{\int_A\tau\dd\mu}$.
By Theorem~\ref{thm:gene}, as $t$ goes to $+\infty$,
\begin{equation}
 \label{eq:TheoremeGeneApplicationFlot}
 \left(\frac{\widetilde{S}_{\lfloor t \rfloor s'} G(\psi)}{\ln (t)}, \, \frac{\widetilde{S}_{\lfloor t \rfloor s'} G(\phi)}{\sqrt{\ln (t)}}\right)_{s'\in[T_0,T_1]}
 \to\left(\int_{\widetilde{A}}g\dd\widetilde{\mu}\, \Phi(0)\Ecal, \, \widetilde \sigma (f) \sqrt{\Phi(0)\Ecal}\, \Ncal\right)_{s'\in[T_0,T_1]}\, ,
\end{equation}
where the convergence is in distribution in $\Ccal([T_1,T_2], \R)$ with respect to $\mu \otimes \delta_0$.

\smallskip

Since $N_t(\cdot)\sim \frac{t}{\int_A\tau\dd\mu}$ as $t\to +\infty$ almost surely for $\mu$,
\begin{equation*}
 \lim_{t \to + \infty} \sup_{s\in[s_1,s_2]}\left|\frac{N_{ts}(\cdot)}{\lfloor t\rfloor}- \frac {s}{\int_A\tau\dd\widetilde{\nu}}\right|
 \to 0
\end{equation*}
$\mu$-almost surely. Thus, still $\mu$-almost surely:
\begin{equation}
 \label{eq:ConvergenceHts}
 \lim_{t \to +\infty} \sup_{s\in[s_1,s_2]}\left|{h_{t,s}(\cdot)}- \frac{s}{\int_A\tau\dd\widetilde{\nu}}\right|
 \to 0\, ,
\end{equation}
with: 
\begin{equation*}
 h_{t,s}(x) = \left\{
 \begin{array}{rcl}
  T_1 & \text{ if } \frac{N_{ts}(x)}{\lfloor t\rfloor} \leq T_1 \\
  \frac{N_{ts}(x)}{\lfloor t\rfloor} & \text{ if } T_1 \leq \frac{N_{ts}(x)}{\lfloor t\rfloor} \leq T_2 \\
  T_2 & \text{ if } T_1 \leq \frac{N_{ts}(x)}{\lfloor t\rfloor}
 \end{array}
 \right. \, .
\end{equation*}
Observe that $h_{t,s}$ takes its values in $[T_1,T_2]$ and is continuous in $s$. 
Therefore, by composition of Equations~\eqref{eq:ConvergenceHts} and~\eqref{eq:TheoremeGeneApplicationFlot}, 
\begin{equation*}
 \left(\frac{\widetilde{S}_{N_{ts}} G(\psi)}{\ln (t)},\frac{\widetilde{S}_{N_{ts}} G(\phi)}{\sqrt{\ln (t)}}\right)_{s\in[s_1,s_2]}
 = \left(\frac{\widetilde{S}_{\lfloor t\rfloor h_{t,s}} G(\psi)}{\ln (t)},\frac{\widetilde{S}_{\lfloor t\rfloor h_{t,s}} G(\phi)}{\sqrt{\ln (t)}}\right)_{s\in[s_1,s_2]}
\end{equation*}
converges in distribution in $\Ccal([s_1,s_2],\R)$, as $t$ goes to $+\infty$, to  
$\left(\int_{\widetilde{A}}G(\psi)\dd\widetilde{\mu}\, \Phi(0)\Ecal, \, \widetilde \sigma (G(\phi)) \sqrt{\Phi(0)\Ecal}\, \Ncal\right)_{s\in[s_1,s_2]}$.

\smallskip

Moreover $\int_{\widetilde{A}} G(\theta) \dd \widetilde{\mu} = \int_{\widetilde{\Mcal}} \theta\dd \widetilde{\nu}$ for $\theta = \psi$, $\psi$ 
by definition of $\widetilde{\nu}$ and of $G(\theta)$, and
\begin{equation*}
 \widetilde{\sigma}^2 (G(\phi))
 =\int_{\widetilde{A}} G(\phi)^2 \dd \widetilde{\mu} + 2\sum_{k \geq 1} \int_{\widetilde{A}} G(\phi) \cdot G(\phi) \circ \widetilde{T}^k \dd \widetilde{\mu}.
\end{equation*}

This finishes the proof of Theorem~\ref{thm:geneflot} for $\psi (x, a, u) := \tau(x)^{-1} \mathbf{1}_0 (a)$ 
and $\mu_0 := \tau^{-1} (x) \dd \mu (x) \otimes \delta_0 (a) \otimes \dd u \in \Pcal (\widetilde{M})$. 
The general case follows from the same ideas as in the proof of Theorem~\ref{thm:gene}: 
\cite[Theorem~1]{Zweimuller:2007} extends the result to any probability measure absolutely continuous 
with respect to $\widetilde{\nu}$, while Hopf's ergodic theorem extends it to any $\psi \in \Lbb^1 (\widetilde{M}, \widetilde{\nu})$.

%Observe that
%%
%\begin{equation*}
% \int_{\widetilde{A}} G(\phi) \cdot G(\phi) \circ \widetilde{T}^k \dd \widetilde{\mu}
% = \int_{\widetilde{\Mcal}}\int_{S_k\tau(x)}^{S_{k+1}\tau(x)}\phi(x,a,u)\phi(\widetilde{Y}_{v}(x,a,0))\dd v\dd\widetilde{\nu}(x,a,u)\, .
%\end{equation*}
%%
%Therefore
%%
%\begin{equation*}
% \sum_{k = 1}^K \int_{\widetilde{A}} G(\phi) \cdot G(\phi) \circ \widetilde{T}^k \dd \widetilde{\mu}
% = \int_{\widetilde{\Mcal}}\int_{\tau(x)}^{S_{K+1}\tau(x)}\phi(x,a,u)\phi(\widetilde{Y}_{v}(x,a,0))\dd v\dd\widetilde{\nu}(x,a,u)\, .
%\end{equation*}
%

\end{proof}

\subsection{Proof for finite horizon Lorentz gases}

We now derive an application to Lorentz gases, that is Corollary~\ref{coroLorentz}, from 
Theorem~\ref{thm:geneflot}.

\begin{proof}[Proof of Corollary~\ref{coroLorentz}]

There exists $c>0$ such that $(\widetilde{\Mcal},c \widetilde{\nu},(\widetilde{Y}_t)_t)$ 
can be represented as a flow as in Theorem~\ref{thm:geneflot}, 
with $(A,\mu,T)$ the corresponding Sinai billiard and $\tau$ the length 
of the free flight until the next collision. Let us write $\Ccal_p$ for the set 
of configurations in $\widetilde{\Mcal}$ whose last reflection is on an obstacle 
corresponding to $A\times \{a\}$. Since $\tau$ is uniformly bounded, the condition 
on $\phi$ ensures that 
\begin{equation*}
 \sum_{a\in\Z^2} \norm{\phi_{|\Ccal_a}}{\eta} <+\infty\, .
\end{equation*}
Again here $\widetilde{T}$ is the billiard transformation in the $\Z^2$-periodic billiard domain.
Let $x$, $y$ in the same continuity domain of $\widetilde{T}$. Then there exists $K$ such that
\begin{align*}
\left| G(\phi) (x) - G(\phi) (y) \right| 
& = \left|\int_0^{\tau(x)} \phi(\widetilde{Y}_s(x,a))\dd s -\int_0^{\tau(y)}\phi(\widetilde{Y}_s(y,a))\dd s\right| \\
& \leq \int_0^{\min\{\tau(x), \tau(y)\}}  \left|\phi(\widetilde{Y}_s(x,a))-\phi(\widetilde{Y}_s(y,a))\right|\dd s + \left|\tau(x)-\tau(y)\right| \norm{\phi_{\Ccal_p}}{\infty} \\
& \leq \norm{\tau}{\infty} \norm{ \phi_{|\Ccal_p}}{\eta} \max_{0 \leq s \leq \min\{\tau(x), \tau(y)\}} d(\widetilde{Y}_s(x,a),\widetilde{Y}_s(y,a))^\eta \\
& \hspace{2em} +\norm{\tau}{\frac{1}{2}} \norm{\phi_{|\Ccal_a}}{\infty}d(x,y)^ {\frac{1}{2}}
\end{align*}
since $\tau$ is $\frac{1}{2}$-H\"older continuous on each continuity component of $T$.
Since $(x,s)\mapsto \widetilde{Y}_s(x,0)$ is differentiable on $\{(x,s)\in A\times[0,+\infty)\, :\, s\leq\tau(x)\}$, 
we conclude that $f:(x,a)\mapsto \int_0^{\tau(x)}\phi(x,a,s)\dd s$ satisfies the assumptions of
Corollary~\ref{cor:billtransfo} with $\eta$ replaced by $\min \{ \eta,1/2 \}$. 

\smallskip

The assumption on the system can be checked as in the proof of Corollary~\ref{cor:billtransfo}: 
\cite[Theorem 3.17]{DemersPeneZhang} ensures that Hypothesis~\ref{hyp:HHH} is satisfied with $p=1$, 
and \cite[Lemma 5.3]{DemersZhang:2014} ensures that $\norm{G(\phi)(\cdot,a)\times}{\Lcal(\Bcal,\Bcal)}\leq C\norm{G(\phi)(\cdot,a)}{\eta}$. 
All is left is to apply Theorem~\ref{thm:geneflot}.
\end{proof}

\section{Limit theorems via induction}\label{sec:induc}

We now prove Proposition~\ref{prop:BillardParInduit} using induced systems as in~\cite{Thomine:2013, Thomine:2015}. 
The strategy, in a nutshell, is as follows. In the present article, up to now, we worked with suspensions flows over an ergodic 
$\Z^2$-extension of a dynamical system $(A, \mu, T)$, where the extension was given by a jump function $F : A \to \Z^2$ and 
the roof function $\widetilde{\tau} : (x,a) \mapsto \tau(x)$. The system $(A, \mu, T)$ 
was a billiard map, and the suspension flow the Lorentz gas.

\smallskip

In~\cite{Thomine:2013, Thomine:2015}, the setting is very similar, with the 
difference that $(A, \mu, T)$ has to be a Gibbs-Markov map (see e.g.~\cite[~4.6]{Aaronson:1997} for an introduction to these systems, which are Markov maps with 
a big image property). Using the symbolic coding of Axiom A flows by Bowen~\cite{Bowen:1973}, 
a statement very close to that of Theorem~\ref{thm:geneflot} was obtained for 
geodesic flows in negative curvature~\cite[Proposition~6.12]{Thomine:2015}. 
The case of Sinai billiards is more complex, as one has to use Young 
towers~\cite{Young:1998} to make them fit the setting of Gibbs-Markov maps.

\subsection{Young towers and Lorentz gas}

To simplify our argument, we shall work with the discrete-time Lorentz gas
(i.e. $\Z^2$-periodic billiard system). 
In order to emphasize the parallel constructions, we keep using the notations 
$(A, \mu, T)$ and $\tau$ in this section, although we stress that they do not 
correspond to the billiard map and the free path length respectively, 
but to an underlying Gibbs-Markov map and to the height of the Young tower. 
Using a Young tower, there exist:
\begin{itemize}
 \item a Gibbs-Markov map $(A, \mu, T)$ with Markov partition $\Gamma$,
 \item a function\footnote{This function $\tau$ is the time of the next Markovian return to the inducing set; 
 it is not the free path length, as it used to be in Subsection~\ref{subsec:bill}.} 
 $\tau : A \to \N_+$ constant on each element of $\Gamma$,
 with $\mu (\tau \geq n) \leq C_\varepsilon e^{- \varepsilon n}$ 
 for some $\varepsilon$, $C_\varepsilon >0$, and a tower $(A_\tau, \mu_\tau, T_\tau)$ over $(A, \mu, T)$ with roof function $\tau$,
 \item a hyperbolic map $(A_Y, \mu_Y, T_Y)$, where each point in $A_Y$ has 
 two coordinates $(x_u, x_s)$ (the base of the Young tower, which has a box structure indexed by $\Gamma$,
 the coordinate $x_u$ is the coordinate along the unstable manifold, and $x_s$ along the stable manifold; we write $\Gamma_Y$ for the corresponding partition of $A_Y$),
 \item a function $\tau_Y : A_Y \to \N_+$ depending only on $x_u$, and a tower $(A_{Y,\tau}, \mu_{Y,\tau}, T_{Y,\tau})$ over $(A_Y, \mu_Y, T_Y)$ with roof function $\tau_Y$,
 \item a factor map $\pi_Y : A_Y \to A$ such that $\tau_Y = \tau \circ \pi_Y$, 
 which lifts to a factor map on the towers: abusing notations, $\pi_Y (x_u, x_s, k) 
  = (\pi_Y (x_u, x_s), k) \in A_\tau$ for all $(x_u, x_s, k) \in A_{Y, \tau}$,
 \item a factor map $\pi$ from $A_{Y,\tau}$ to the Sinai billiard table.
\end{itemize}

These objects behave well when one works with $\Z^2$-extensions. Let $F_L$ be the function 
describing the jumps for the discrete-time Lorentz gas
(i.e.\  $F_L$ is the function denoted $F$ in Subsection~\ref{subsec:bill}) 
and $F_Y (x_u, x_s) = \sum_{k=0}^{\tau_Y (x_u)-1} F_L \circ \pi (x_u, x_s, k)$. By the construction of the Young tower, $F_L \circ \pi$ depends only on $x_u$, 
and thus quotients through $\pi_Y$ to yield $F : A \to \Z^2$, which is constant on each element of $\Gamma$.

\smallskip

Let $(\widetilde{A}_{Y,\tau}, \widetilde{\mu}_{Y,\tau}, \widetilde{T}_{Y,\tau})$ 
be the system defined by:
\begin{itemize}
 \item $\widetilde{A}_{Y,\tau} = A_{Y,\tau} \times \Z^2$,
 \item $\widetilde{\mu}_{Y,\tau} = \sum_{a \in \Z^2} \mu_{Y,\tau} \otimes \delta_a$,
 \item $\widetilde{T}_{Y,\tau} (x_u, x_s, k, a) = (x_u, x_s, k+1,a)$ if $k < \tau_Y (x_u)-1$, 
 and otherwise $\widetilde{T}_{Y,\tau} (x_u, x_s,\tau(x_u)-1,a) = (T_Y (x_u, x_s), 0,a+F_Y(x_u))$.
\end{itemize}
In the same way, define $(\widetilde{A}_\tau, \widetilde{\mu}_\tau, \widetilde{T}_\tau)$ 
using the system $(A_\tau, \mu_\tau, T_\tau)$ and the function $F$. Then there exist two 
factor maps $\widetilde{\pi}$ and $\widetilde{\pi}_Y$ from $(\widetilde{A}_{Y,\tau}, \widetilde{\mu}_{Y,\tau}, \widetilde{T}_{Y,\tau})$, descending to the discrete-time Lorentz gas (i.e. the $\Z^2$-periodic 
billiard system $\left(\widetilde A,\widetilde\mu,\widetilde T\right)$ defined in Subsection~\ref{subsec:bill}) and 
to $(\widetilde{A}_\tau, \widetilde{\mu}_\tau, \widetilde{T}_\tau)$ respectively.
This construction is summed up in the following diagram:

\begin{equation*}
\begin{tikzcd}
    (\widetilde{A}_\tau, \widetilde{\mu}_\tau, \widetilde{T}_\tau) \arrow{d}{} & (\widetilde{A}_{Y,\tau}, \widetilde{\mu}_{Y,\tau}, \widetilde{T}_{Y,\tau}) \arrow{l}{\widetilde{\pi}_Y} \arrow{r}{\widetilde{\pi}} \arrow{d}{} & \left(\substack{\text{collision map for} \\ \text{the Lorentz gas}}\right) \arrow{d}{}\\
    (A_\tau, \mu_\tau, T_\tau) & (A_{Y,\tau}, \mu_{Y,\tau}, T_{Y,\tau}) \arrow{l}{\pi_Y} \arrow{r}{\pi} & \left(\substack{\text{collision map for} \\ \text{the Sinai billiard}}\right)
\end{tikzcd}
\end{equation*}
In the diagram above, all the downward arrows consist in forgetting the $\Z^2$-coordinate, all the horizontal arrows are measure-preserving, 
and $\widetilde{\pi}_Y$ (but not $\widetilde{\pi}$) acts trivially on the $\Z^2$-coordinate. 

\smallskip

We shall also write, for $x \in A$:
\begin{align*}
 \varphi (x) & := \inf\{n \geq 1 : \ S_n^T F (x) = 0\}\, , \\
 \widetilde{\varphi} (x) & := \sum_{k=0}^{\varphi(x)-1} \tau \circ T^k (x)\, ,
\end{align*}
so that $\varphi$ is the first return time to $A \times \{0\}$ for the underlying $\Z^2$-extension 
of a Gibbs-Markov map, and $\widetilde{\varphi}$ the first return time to $A \times \{0\} \times \{0\}$ 
for $\widetilde{T}_\tau$. Then the map $\widetilde{T}_0 := \widetilde{T}_\tau^{\widetilde{\varphi}}$ acts on 
$A \times \{0\} \simeq A$, and $(A, \mu, \widetilde{T}_0)$ is a measure-preserving ergodic Gibbs-Markov map for 
some refined partition $\Gamma_0$. In the same way, we define $\widetilde{T}_{Y,0} := \widetilde{T}_{Y,\tau}^{\widetilde{\varphi}}$.

\smallskip

Given an observable $f$ defined on the state space of the Lorentz gas ($\Z^2$-periodic billiard map), we define the sum of $f$ 
along an excursion, either until it comes back to the base of the Young tower or to the basis of 
the cell $0$ in $\widetilde{A}_{Y,\tau}$. For $(x_u, x_s) \in A_Y$ and $a \in \Z^2$, let:
\begin{align*}
 G_{Y, \tau} (f) (x_u, x_s,a) & := \sum_{k=0}^{\tau_Y (x_u)-1} f (x_u, x_s,k,a) \\
 G_{Y, \varphi} (f) (x_u, x_s) & := \sum_{k=0}^{\widetilde{\varphi} (x_u)-1} f \circ \widetilde{T}_{Y, \tau}^k (x,0,0) \\
 & = \sum_{k=0}^{\varphi (x_u)-1} G_{Y, \tau} (f) (T_Y^k (x_u, x_s),S_k^{T_Y} F_Y (x_u))\, ,
\end{align*}
and define in the same way $G_\tau (f):A\times\Z^2\rightarrow \mathbb C$ and $G_\varphi (f):A\rightarrow\mathbb C$
for functions $f$ defined on $\widetilde{A}_\tau$.

\subsection{Proof of Proposition~\ref{prop:BillardParInduit}}

The general strategy, close to that of~\cite[Proposition~6.12]{Thomine:2015}, is as follows:
\begin{itemize}
 \item Take a function $f$ defined on the state space of the discrete-time Lorentz gas, 
 uniformly $\eta$-H\"older on the continuity components of the billiard map, 
 with integral zero and such that $\sum_{a \in \Z^2} (1+\ln_+ |a| )^{\frac{1}{2}+\varkappa} \norm{f(\cdot, a)}{\infty} < + \infty$ 
 for some $\varkappa >0$. Lift it to a function $f \circ \widetilde{\pi}$ defined on $(\widetilde{A}_{Y,\tau}, \widetilde{\mu}_{Y,\tau}, \widetilde{T}_{Y,\tau})$.
 \item Add a bounded coboundary $u \circ T - u$ to get $f_+ \circ\widetilde \pi_Y = f \circ \widetilde{\pi} + u \circ T - u$  
(independent from $x_s$ and thus going to the quotient through $\widetilde{\pi}_Y$, 
 so that we only need to work with the Gibbs-Markov extension). 
 \item Check that $G_{Y,\tau} (f \circ \widetilde{\pi})$ satisfies some integrability conditions, 
 then apply~\cite[Lemma~4.16]{PeneThomine:2019} and \cite[Lemma~2.7]{PeneThomine:2019} to 
 show that $G_\varphi (f_+)$ is also integrable enough (the precise conditions shall be described later).
 \item Apply a version of~\cite[Corollary~6.10]{Thomine:2015}, together with \cite[Remark~4.6]{Thomine:2014}, which states:
\end{itemize}

\begin{proposition}[\cite{Thomine:2015}]
\label{prop:2015}
 
 Let $(\widetilde{A}_\tau, \widetilde{\mu}_\tau, \widetilde{T}_\tau)$ be an ergodic 
 and recurrent Markov $\Z^2$-extension of a Gibbs-Markov map $(A, \Gamma, \mu, T)$, 
 of roof function $\tau$ and of step function $F: A \to \Z^2$. Assume that it is aperiodic, 
 that $\tau$ and $F$ belong to $\Lbb^2 (A, \mu_A)$, and that $\sum_{\gamma \in \Gamma} \mu(\gamma) |\tau|_{\Lip (\gamma)}$ 
 is finite. Under these hypotheses, the covariance matrix $\Sigma^2 (F)$ is positive definite, where, 
 for all $u$ and $v$ in $\R^2$:
 \begin{equation*}
  \left( u, \Sigma^2 (F) v \right) 
  = \lim_{N \to + \infty} \frac{1}{N} \int_A \left( \sum_{i=0}^{N-1} F \circ T^i, u \right) \left( \sum_{i=0}^{N-1} F \circ T^i, v \right) \dd \mu.
 \end{equation*}

 Let $f_+$ be a real-valued, measurable function from $\widetilde{A}_\tau$ to $\R$. Assume that:
 \begin{itemize}
  \item $\sup_{0 \leq n \leq \widetilde{\varphi} (x)} \left| \sum_{k=0}^{n-1} f_+ \circ \widetilde{T}_\tau^k (\cdot,0,0) \right| \in \Lbb^q (A, \mu)$ for some $q > 2$,
  \item $\int_A G_\varphi (f_+) \dd \mu = 0$,
  \item $\sum_{\gamma \in \Gamma_0} \mu (\gamma) \sup_{a \in \Z^d} |G_\tau (f_+) (\cdot, a)|_{\Lip (\gamma)}$ is finite.
 \end{itemize}
 Then, for any probability measure $\widetilde{\nu}$ absolutely continuous with respect to $\widetilde{\mu}_\tau$: 
 \begin{equation*}
 \left( \frac{2 \pi \sqrt{\det (\Sigma^2 (F))}}{\ln (n)} \right)^{\frac{1}{2}} \sum_{k=0}^{n-1} f_+ \circ \widetilde{T}^k
 \to \sigma (f_+) L,
 \end{equation*}
 where the convergence is in distribution when the left-hand side is seen as a random variable from $(\widetilde{A}_\tau, \widetilde{\nu})$ 
 to $\R$, where $L$ follows a centered Laplace distribution of variance $1$, and:
 \begin{equation*}
  \sigma^2 (f_+)
  = \int_A G_\varphi (f_+)^2 \dd \mu + 2 \sum_{n=1}^{+\infty} \int_A G_\varphi (f_+) \cdot G_\varphi (f_+) \circ \widetilde{T}_0^n \dd \mu,
 \end{equation*}
 where the limit is taken in the Ces\`aro sense.
\end{proposition}

\begin{proof}[Proof of Proposition~\ref{prop:BillardParInduit}]

Let us go through the assumptions of Proposition~\ref{prop:2015} for the Young towers associated with Sinai billiards. 

\medskip
\textbf{General assumptions on the system}
\smallskip

The bidimensional Lorentz gas is ergodic and recurrent for the Liouville measure; by construction, 
so is $(\widetilde{A}_{Y,\tau}, \widetilde{\mu}_{Y,\tau}, \widetilde{T}_{Y,\tau})$. As a 
factor map, $(\widetilde{A}_\tau, \widetilde{\mu}_\tau, \widetilde{T}_\tau)$ is then also 
ergodic and recurrent.

\smallskip

The theorem stays true if one drops the hypothesis of aperiodicity on the extension; 
the full reduction can be found in the proof of Proposition~2.11 in~\cite{PeneThomine:2019}.

%The argument used in the proof of  shows that the 
%aperiodicity assumption can be dropped. Indeed, a $\Z^2$-extension of $(A, \mu, T)$ is also 
%a $\Z^2$-extension of $T_\Lambda (x,a) := (T(x), a+F(x) [\Lambda])$ acting on $A \times (\Z^2/\Lambda)$ 
%for any rank $2$ lattice $\Lambda < \Z^2$. For well-chosen $\Lambda$, this new extension is aperiodic, 
%which allows for the application of Proposition~\ref{prop:2015}. The assumptions on $f$ 
%are still valid for this new extension. An additional argument, 
%provided in~\cite{PeneThomine:2019}, shows that the constants $\det(\Sigma^2 (F))$ and $\sigma (f)$ 
%can still be taken as in Proposition~\ref{prop:2015}.

\smallskip

The roof function $\tau_Y$ has an exponential tail, and as such belongs to $\Lbb^{2+\varepsilon} (A, \mu)$. Since the billiard has 
finite horizon, the function $F_L$ is uniformly bounded. Hence, the size of the jumps $F_Y  = S_{\tau_Y} (F_L \circ \pi)$ 
is in $O(\tau_Y)$, and thus also belongs to $\Lbb^{2+\varepsilon} (A, \mu)$. By construction of the Young towers, $\tau_Y$ 
is constant on the elements of the Markov partition. All these properties goes through the quotient 
to $\tau$, and in particular $\sum_{\gamma \in \Gamma} \mu(\gamma) |\tau|_{\Lip (\gamma)} = 0$.

\medskip
\textbf{Defining a coboundary}
\smallskip

%\textcolor{red}{Bien vu, j'ai modifie pas mal de choses. Je suis reparti de l'article de Young, mais cela entraine quelques complications dans l'argument a la Bowen, que je dois re-detailler. Les modifications associees sont en vert.}

Let $f$ be an observable of the collision section for the Lorentz gas which is uniformly $\eta$-H\"older 
on the continuity sets of the billiard map. The space $A_{\tau,Y}$ has a box structure, with, by Young's 
construction, a distinguished piece of unstable manifold on the basis $A_Y$. Let us choose the coordinates 
$(x_u, x_s)$ so that this piece of unstable manifold is $\{x_s=0\}$. Then we get a map:
\begin{equation*}
 p_+ : \left\{ 
 \begin{array}{lll}
  \widetilde{A}_{Y, \tau} & \to & \{(x_u, 0, k,a) : \ 0 \leq k < \tau (x_u), \ a \in \Z^2\} \\
  (x_u, x_s, k, a) & \mapsto & (x_u, 0, k,a)
 \end{array}
 \right. ,
\end{equation*}

The space $A_{Y, \tau}$ is also endowed with a distance $d_{Y, \tau}$ 
satisfying the properties (P3) and (P4a) in~\cite{Young:1998}, namely, there exist constants 
$C > 0$ and $\alpha \in (0,1)$ such that, for all $x_u$, $x_s$, $x_u'$, $x_s'$:
\begin{itemize}
 \item $d_{Y, \tau} (T_{Y,\tau}^n (x_u, x_s, 0), T_{Y,\tau}^n (x_u, x_s', 0)) \leq C \alpha^n$ (contraction along stable leaves),
 \item $d_{Y, \tau} (T_{Y,\tau}^n (x_u, x_s, 0), T_{Y,\tau}^n (x_u', x_s, 0)) \leq C \alpha^{s_0(x_u,x_u')-n}$ for $0 \leq n < s_0 (x_u, x_u')$ 
(backward contraction along unstable leaves),
\end{itemize}
where $s_0$ is a separation time. In addition, up to working with some power of $d_{Y, \tau}$, we may assume that 
$f \circ \widetilde{\pi} (\cdot, \cdot, \cdot, a)$ is Lipschitz for $d_{Y, \tau}$ uniformly in $a \in \Z^2$.

\smallskip

The function $f \circ \widetilde{\pi}$ is defined on $A_{Y,\tau} \times \Z^2$.
To get a function defined on $A_\tau \times \Z^2$, we use a classical trick by Bowen~\cite{Bowen:1975}, 
also used in the proof of~\cite[Proposition~6.12]{Thomine:2015}. While we shall not repeat 
the computations, let us ouline the main arguments. Define:
\begin{equation*}
 u(x_u, x_s, k, a) 
 := \sum_{n = 0}^{+ \infty} \left[ f \circ \widetilde{\pi} \circ \widetilde{T}_{Y,\tau}^n (x_u, x_s, k,a) - f \circ \widetilde{\pi} \circ \widetilde{T}_{Y,\tau}^n \circ p_+ (x_u, x_s, k,a) \right] \, .
\end{equation*}
The function $u$ is zero on $\{x_s=0\}$, and the contraction along stable leaves implies that $u$ is bounded. 
The function $f_+ := f \circ \widetilde{\pi} + u \circ \widetilde{T}_{Y,\tau} - u$ 
also does not depend on the $x_s$ coordinate. Abusing notations, we may see $f_+$ as defined on $\widetilde{A}_\tau$. 
Finally, there exists a constant $C'$ such that, for all $x_u$, $x_u'$ in the same element of $\Gamma$, for all $x_s$ and all $a \in \Z^2$:
\begin{equation}
 \label{eq:RegulariteU}
 | u(x_u, x_s, 0, a) - u(x_u', x_s, 0, a) | 
 \leq C' \alpha^{\frac{s_0(x_u, x_u')}{2}}. 
\end{equation}

Since $S_n f = S_n f_+ \circ \widetilde{\pi} + u \circ \widetilde{T}_{Y,\tau}^n - u$ and $u$ is bounded, it is enough to 
prove the convergence in distribution 
\begin{equation*}
 \left( \frac{2 \pi \sqrt{\det (\Sigma^2 (F))}}{\ln (n)} \right)^{\frac{1}{2}} \sum_{k=0}^{n-1} f_+ \circ \widetilde{T}_\tau^k
 \to \sigma (f_+) \sqrt{\Ecal} \Ncal,
\end{equation*}
with respect to the probability distribution $\mu \otimes \delta_0 \otimes \delta_0 \in \Pcal (\widetilde{A}_\tau)$. 
The convergence
\begin{equation*}
 \left( \frac{2 \pi \sqrt{\det (\Sigma^2 (F))}}{\ln (n)} \right)^{\frac{1}{2}} \sum_{k=0}^{n-1} f \circ \widetilde{\pi} \circ \widetilde{T}_{Y, \tau}^k
 \to \sigma (f_+) \sqrt{\Ecal} \Ncal,
\end{equation*}
with respect to $\mu_Y \otimes \delta_0 \otimes \delta_0 \in \Pcal (\widetilde{A}_{Y,\tau})$ then follows, 
and the convergence with respect to any absolutely continuous probability measure on $\widetilde{A}_{Y,\tau}$ 
follows from~\cite[Theorem~1]{Zweimuller:2007}. In addition, since $f_+ - f \circ \widetilde{\pi}$ is a bounded 
coboundary and adding a bounded coboundary does not change the asymptotic variance in the central limit theorem,
\begin{equation}
 \label{eq:VarianceInduit}
 \sigma (f_+) 
 = \int_{A_Y} G_{Y, \varphi} (f\circ \widetilde{\pi})^2 \dd \mu_Y + 2 \sum_{n=1}^{+\infty} \int_{A_Y} G_{Y, \varphi} (f\circ \widetilde{\pi}) \cdot G_{Y, \varphi} (f\circ \widetilde{\pi}) \circ \widetilde{T}_{Y, 0}^n \dd \mu_Y.
\end{equation}

All is left is to check the integrability and regularity assumptions on $f_+$.

\medskip
\textbf{Integrability of $f_+$}
\smallskip

We start with the first condition on $f_+$ in Proposition~\ref{prop:2015}, which is the hardest. Since
\begin{align*}
 \sup_{0 \leq n \leq \widetilde{\varphi} (x)} \left| \sum_{k=0}^{n-1} f_+ \circ \widetilde{T}_\tau^k (\cdot,0,0) \right| 
 & = \sup_{0 \leq n \leq \widetilde{\varphi} (x)} \left| \sum_{k=0}^{n-1} (f \circ \widetilde{\pi} + u \circ \widetilde{T}_{Y,\tau} -u) \circ \widetilde{T}_{Y,\tau}^k (\cdot,0,0) \right| \\
 & \leq \sup_{0 \leq n \leq \widetilde{\varphi} (x)} \left| \sum_{k=0}^{n-1} f \circ \widetilde{\pi} \circ \widetilde{T}_{Y,\tau}^k (\cdot,0,0)\right| + 2 \norm{u}{\infty} \\
 & \leq G_{Y, \varphi} (|f \circ \widetilde{\pi}|) + 2 \norm{u}{\infty},
\end{align*}
it is enough to check that $G_{Y, \varphi} (|f \circ \widetilde{\pi}|) \in \Lbb^q (A_Y, \mu_Y)$ for some $q>2$. 
For $(x_u, a) \in A \times \Z^2$, let:
\begin{equation*}
 \overline{f} (x_u, k,a) 
 := \sup_{x_s} |f \circ \widetilde{\pi}| (x_u, x_s, k,a).
\end{equation*}
Then $G_{Y, \varphi} (|f \circ \widetilde{\pi}|) (x_u, x_s) \leq G_\varphi (\overline{f}) (x_u)$, so 
it is enough to check that $G_\varphi (\overline{f}) \in \Lbb^q (A_Y, \mu_Y)$ for some $q>2$.

\smallskip

For all $a \in \Z^2 \setminus \{0\}$, let
\begin{equation*}
 N_a (x_u) 
 := G_\varphi (\mathbf{1}_{(A \times \{0\} \times \{a\})})
\end{equation*}
be the number of times an excursion from $A \times \{0\} \times \{0\}$ hits the basis of the Young tower 
at $A \times \{0\} \times \{a\}$ before going back to $A \times \{0\} \times \{0\}$. Let $A_a := \{N_a \neq 0\} \subset A$, 
and 
\begin{equation*}
 \mu_a 
 := \mu(A_a)^{-1} \widetilde{T}_{A \times \{0,a\}*} \mu_{|A_a} \otimes \delta_0 \, ,
\end{equation*}
where $\widetilde{T}_{A \times \{0,a\}}$ is the map induced by $\widetilde{T}$ on $A \times \{0, a\}$.
In other words, $\mu_a \in \Pcal (A \times \{a\}) \simeq \Pcal (A)$ is the distribution of a point 
at which a trajectory starting from $A \times \{0\}$ enters $A \times \{a\}$, conditioned by the fact 
that this trajectory enters $A \times \{a\}$ before going back to $A \times \{0\}$. Then the 
distribution of $N_a-1$ for $\mu (\cdot |A_a)$ is the distribution of the first non-negative 
hitting time $\varphi_{-a}$ of $A_{-a}$ for $\mu_a$.

\smallskip

By~\cite[Lemma~4.8]{PeneThomine:2019}, the densities $\de \mu_a / \de \mu$ are in $\Lbb^\infty (A, \mu)$ 
and uniformly bounded in $a$. We apply~\cite[Lemma~4.16]{PeneThomine:2019} to the family of measures $(\mu_a)_{a \in \Z^2 \setminus \{0\}}$ 
and the function $\mathbf{1}_A$. Note that $\alpha (a) = \mu(A_a) = \mu(A_{-a})$ in the cited article. 
Hence, for all $q \in (2, \infty)$, there exists a constant $C > 0$ such that, for all $a\in\Z^2$:
\begin{align*}
 \norm{G_\varphi (\overline{f} \mathbf{1}_{A_\tau \times \{a\}})}{\Lbb^q (A, \mu)} 
 & = \mu (A_a)^{\frac{1}{q}} \norm{G_\varphi (\overline{f} \mathbf{1}_{A_\tau \times \{a\}})}{\Lbb^q (A, \mu (\cdot |A_a))} \\
 & = \mu (A_a)^{\frac{1}{q}} \norm{\sum_{k=0}^{\varphi_{-a} (x)} G_\tau (\overline{f}) (\widetilde{T}_0^k (x), a)}{\Lbb^q (A, \mu_a)} \\
 & \leq C\mu (A_a)^{\frac{1}{q}-1} \norm{G_\tau (\overline{f}) (\cdot, a)}{\Lbb^q (A, \mu)}.
\end{align*}
By~\cite[Corollary~2.9]{PeneThomine:2019} and~\cite[Proposition~2.6]{PeneThomine:2019}, 
with $\alpha=d=2$ and $L\equiv 1$, 
\begin{equation*}
 \mu (A_a) 
 = \Theta \left(\frac{1}{1+\ln_+ |a|} \right).
\end{equation*}
Hence, up to taking a larger constant $C$,
\begin{equation}
 \label{eq:MinkowskiNoyauPotentiel}
 \norm{G_\varphi (\overline{f})}{\Lbb^q (A, \mu)} 
 \leq \sum_{a \in \Z^2} \norm{G_\varphi (\overline{f} \mathbf{1}_{A_\tau \times \{a\}})}{\Lbb^q (A, \mu)} 
 \leq C \sum_{a \in \Z^2} (1+\ln_+ |a|)^{1-\frac{1}{q}} \norm{G_\tau (\overline{f}) (\cdot, a)}{\Lbb^q (A, \mu)}.
\end{equation}
In addition, focusing on a single term $\norm{G_\tau (\overline{f}) (\cdot, a)}{\Lbb^q (A, \mu)}$, we get:
\begin{align}
 \norm{G_\tau (\overline{f}) (\cdot, a)}{\Lbb^q (A, \mu)} 
 & \leq \norm{\sum_{r\ge 1} \mathbf{1}_{\{\tau=r\}}\sum_{k=0}^{r-1} \norm{f (\cdot, a+S_k F)}{\Lbb^\infty(\{\tau=r\})}}{\Lbb^q (A, \mu)} \nonumber \\
 & = \norm{\sum_{r\ge 1}\mathbf{1}_{\tau=r} \sum_{k=0}^{r-1} \sum_{a' \in \Z^2} \norm{f (\cdot, a')}{\infty} \mathbf{1}_{a' = a+S_k F} }{\Lbb^q (A, \mu)} \nonumber \\
 & \leq \sum_{a' \in \Z^2} \norm{f (\cdot, a')}{\infty} \sum_{r\ge 1} \sum_{k=0}^{r-1} \mu (\tau=r,\, S_k F = a'-a)^{\frac{1}{q}}. \label{eq:Convolution}
\end{align}
Set $h_q (a) := C (1+\ln_+ |a|)^{1-\frac{1}{q}}$ and $g_q (a) := \sum_{r\ge 1} \sum_{k=0}^{r-1} \mu (\tau=r,\, S_k F =
 a)^{\frac{1}{q}}$. 
%%
%\begin{align}
% \norm{G_\tau (\overline{f}) (\cdot, a)}{\Lbb^q (A, \mu)}  & \leq \norm{\sum_{\gamma \in \Gamma} \mathbf{1}_\gamma\sum_{k=0}^{\tau (\gamma)-1} \norm{f (\cdot, a+S_k F)}{\Lbb^\infty(\gamma)}}{\Lbb^q (A, \mu)} \nonumber \\
% & = \norm{\sum_{\gamma \in \Gamma}\mathbf{1}_\gamma \sum_{k=0}^{\tau (\gamma)-1} \sum_{a' \in \Z^2} \norm{f (\cdot, a')}{\infty} \mathbf{1}_{a' = a+S_k F} }{\Lbb^q (A, \mu)} \nonumber \\
% & \leq \sum_{a' \in \Z^2} \norm{f (\cdot, a')}{\infty} \sum_{\gamma \in \Gamma} \sum_{k=0}^{\tau (\gamma)-1} \mu (\gamma\cap\{S_k F = a'-a\})^{\frac{1}{q}}. \label{eq:Convolution}
%\end{align}
%%
%Set $h_q (a) := C (1+\ln_+ |a|)^{1-\frac{1}{q}}$ and $g_q (a) := \sum_{\gamma \in \Gamma} \sum_{k=0}^{\tau (\gamma)-1} \mu (\gamma\cap\{ S_k F = a\})^{\frac{1}{q}}$. 
Equations~\eqref{eq:MinkowskiNoyauPotentiel} and~\eqref{eq:Convolution} together imply that:
\begin{equation}
 \label{eq:BorneGTau}
 \norm{G_\varphi (\overline{f})}{\Lbb^q (A, \mu)} \leq \sum_{a \in \Z^2} (h_q * g_q) (a) \norm{f (\cdot, a)}{\infty}\, .
\end{equation}
If $S_k F = a$ with $k\leq r-1$, then $r \geq k \geq |a|/\norm{F}{\infty}$. Since $\mu (\tau \geq k) \leq C_\varepsilon e^{- \varepsilon k}$, 
there exists a constant $C' (q, \varepsilon)$ such that:
\begin{equation*}
 g_q (a)
 \leq \sum_{r\geq |a|/\norm{F}{\infty} }r\, \mu(\tau=r)
 \leq C' (q, \varepsilon)e^{-\frac{\varepsilon |a|}{2q \norm{F}{\infty}}}.
\end{equation*}
All is left is to estimate $h_q * g_q$. Let $a \in \Z^2 \setminus \{0\}$. 
We split $\Z^2$ into rings:
\begin{equation*}
 A_n (a) 
 = \{a' \in \Z^2: \ e^n |a| \leq |a'| < e^{n+1} |a| \},
\end{equation*}
with $n \geq 1$, and a central disk $A_0 (a)$. We have $\Card (A_n (a)) = \Theta (e^{2n} |a|^2)$ and, for all $a' \in A_n (a)$,
\begin{equation*}
 \left\{ 
 \begin{array}{lll}
  h_q (a') & \leq & h_q (a) + C (1-q^{-1}) (n+1), \\
  g_q (a-a') & \leq & C' (q, \varepsilon) e^{-\frac{\varepsilon (e^n-1) |a|}{2 q \norm{F}{\infty}}}.
 \end{array}
 \right.
\end{equation*}
Summing over all $a' \in \Z^2$ yields, for some constant $C' >0$:
\begin{align}
 h_q * g_q (a) 
 & = \sum_{n = 0}^{+ \infty} \sum_{a' \in A_n (a)} h_q (a') g_q (a-a') \nonumber \\
 & \leq \sum_{a' \in \Z^2} h_q (a) g_q (a-a') + C (1-q^{-1}) \sum_{a' \in A_0 (a)} g_q (a-a') + C (1-q^{-1}) \sum_{n = 1}^{+ \infty} (n+1) \sum_{a' \in A_n (a)} g_q (a-a') \nonumber \\
 & \leq [h_q (a) + C (1-q^{-1})] \norm{g_q}{\ell^1 (\Z^2)} + C' \sum_{n = 1}^{+ \infty} (n+1) e^{2n} |a|^2 e^{-\frac{\varepsilon (e^n-1) |a|}{2 q \norm{F}{\infty}}}. \label{eq:DecoupageEnAnneaux}
\end{align}
The sum in Equation~\eqref{eq:DecoupageEnAnneaux} is finite for all $a$. Each term in the sum converges to $0$ 
as $a$ goes to infinity (and thus is bounded). In addition, the function $u \mapsto  u^2 e^{-\frac{\varepsilon (e^n-1) u}{2q \norm{F}{\infty}}}$ 
is decreasing on $[(4q\norm{F}{\infty})/(e^n-1),+\infty)$, and thus on $[1,+\infty)$ for all large enough $n$.
Hence, for all large enough $n$ and all $a \in \Z^2 \setminus \{0\}$,
\begin{equation*}
 (n+1) e^{2n} |a|^2 e^{-\frac{\varepsilon (e^n-1) |a|}{2q \norm{F}{\infty}}} 
 \leq (n+1) e^{2n} e^{-\frac{\varepsilon (e^n-1)}{2q \norm{F}{\infty}}},
\end{equation*}
which is summable in $n$. Hence the sum is bounded in $a$. Since $h_q$ is bounded from below, we finally get $h_q * g_q  = O(h_q)$.

\smallskip

Let $\varkappa >0$, and $f$ be such that $\sup_{a \in \Z^2} (1+\ln_+ |a|)^{\frac{1}{2}+\varkappa} \norm{f (\cdot, a)}{\infty} < +\infty$. 
Without loss of generality, we assume that $\varkappa < 1/2$. Taking $q = \frac{2}{1-2 \varkappa}$, by Equation~\eqref{eq:BorneGTau}, 
the function $G_\varphi (\overline{f})$ belongs to $\Lbb^q (A, \mu)$.

\medskip
\textbf{Remaining conditions on $f_+$}
\smallskip

Let us focus on the last two conditions for $G_\varphi (f)$. Since $f$ is integrable and has integral zero, so does 
$f \circ \widetilde{\pi}$. By Kac's formula, $G_{Y, \varphi} (f)$ is integrable and:
\begin{equation*}
 \int_{A_Y} G_{Y, \varphi} (f) \dd \mu_Y 
 = \int_{\widetilde{A}_{\tau,Y}} f \dd \widetilde{\mu}_{Y, \tau} 
 = 0.
\end{equation*}
Since $G_\varphi (f_+) - G_{Y, \varphi} (f)$ is a bounded coboundary, 
$G_\varphi (f_+)$ also has integral zero.

\smallskip

Finally, let us check the regularity condition on $G_\varphi (f_+)$. Summing the identity $f_+ := f \circ \widetilde{\pi} + u \circ \widetilde{T}_{Y,\tau} - u$ on the 
height of the tower $A_{Y, \tau}$ yields, for all $x_u \in A$ and $a \in \Z^2$,
\begin{equation*}
 G_\tau (f_+) (x_u, a) 
 = G_\tau (f \circ \widetilde{\pi}) (x_u, 0, a) + u (T_Y(x_u,0), 0, a+F(x_u)) - u (x_u, 0, 0, a).
\end{equation*}
The space $A$ can be endowed with a metric $\alpha^s$, where $s$ is the separation time 
for the Gibbs-Markov map $(A, \mu, T)$ and $\alpha \in (0,1)$ is close enough to $1$. 
As $s \leq s_0$, we have $\alpha^{s_0} \leq \alpha^s$, and $\alpha^{s_0-\tau} \leq \alpha^{-1} \alpha^s$ 
if $s_0 \geq \tau$ (so on each element of the partition $\Gamma$). Given $x_u$, $x_u'$ in 
the same element of $\Gamma$,
\begin{align*}
 |G_\tau (f \circ \widetilde{\pi}) (x_u, 0, a)-G_\tau (f \circ \widetilde{\pi}) (x_u', 0, a)| 
 & \leq C |f \circ \widetilde{\pi}|_{\Lip (d_{Y, \tau})} \sum_{k=0}^{\tau(x_u)-1} \alpha^{s_0 (x_u, x_u')-k} \\
 & \leq \frac{C \alpha}{1-\alpha} |f \circ \widetilde{\pi}|_{\Lip (d_{Y, \tau})} \alpha^{s_0 (x_u, x_u')-\tau (x_u)} \\
 & \leq \frac{C}{1-\alpha} |f \circ \widetilde{\pi}|_{\Lip (d_{Y, \tau})} \alpha^{s (x_u, x_u')},
\end{align*}
so the function $x_u \mapsto G_\tau (f \circ \widetilde{\pi}) (x_u, 0, a)$ is Lipschitz 
for the distance $\alpha^s$ on each element of $\Gamma$, uniformly in $a \in \Z^2$ and in $\Gamma$.
\smallskip

By Equation~\eqref{eq:RegulariteU}, the function $u$ is uniformly $1/2$-H\"older for the distance 
$\alpha^{s_0}$ (and thus for the distance $\alpha^s$) on each unstable leaf in $A_Y \times \Z^2$. 
Up to increasing the value of $\alpha$, we may assume that $u$ is actually Lipschitz. Since applying $T_Y$ multiplies 
$\alpha^s$ by at most $\alpha^{-1}$, the function $x_u \mapsto u (T_Y(x_u, 0), 0, a+F(x_u))$ 
is also Lipschitz for the distance $\alpha^s$ on each element of $\Gamma$, 
uniformly in $a \in \Z^2$ and in $\Gamma$. Hence, $f_+$ is also Lipschitz for the 
distance $\alpha^s$ on each element of $\Gamma$, uniformly in $a \in \Z^2$ and in $\Gamma$, 
and thus $G_\varphi (f_+)$ satisfies the regularity condition of Proposition~\ref{prop:2015} 
by~\cite[Lemma~6.5]{Thomine:2015}.
\end{proof}

\begin{remark}[Infinite horizon billiards]

 Young towers are still available for infinite horizon Lorentz gases~\cite{ChernovZhang:2005}, 
although the height of the tower only has a polynomial tail: $\mu (\tau \geq n) = O(n^{-2})$. 
In the finite horizon case, we used the facts that jumps in the billiard are bounded and 
that the tails of $\tau$ decay exponentially to control $g_q$; both fail in the infinite horizon setting.

\smallskip

Moreover it is not known whereas a spectral local limit theorem analogous to Condition~\eqref{hyp:0} 
holds (with $\ell$ replaced by ${\ell\log \ell}$) in the infinite horizon setting. Using Young towers, 
in~\cite{SzaszVarju:2007}, Sz\'asz and Varj\'u proved estimates analogous to our Hypothesis~\ref{hyp:HHH} 
under a weaker form, namely with $\Lcal(\Bcal,\Bcal)$ replaced by $\Lcal(\Bcal,\Lbb^1(\mu))$.
However, Condition~\eqref{hyp:0} with $\Lcal(\Bcal,\Bcal)$ being replaced by $\Lcal(\Bcal,\Lbb^1(\mu))$ 
would not be enough to adapt our proof of Theorem~\ref{thm:gene}, because we use iterations of operators
$Q_{\ell,a}:h\mapsto P^\ell \left(\mathbf{1}_{\{S_\ell F=a\}}\, h\right)$.
\end{remark}

\section*{Acknowledgments}
This research has been done mostly in the Mathematical Departments of Brest and Orsay Universities, and also at the Institut Henri Poincar\'e that we thank for their hospitalities.
FP is grateful to the Institut Universitaire de France (IUF) for its important support.

\end{document}